\documentclass[10pt]{amsart}
\RequirePackage[a4paper, includehead, headheight=0.6cm, inner=2.0cm ,outer=2.0cm, top=2.5cm, bottom=2.5cm]{geometry}

\usepackage{graphicx} %
\usepackage{amsmath}
\usepackage{xcolor}
\numberwithin{equation}{section}
\usepackage{amssymb}
\usepackage{thmtools}
\usepackage{amsthm}
\usepackage{graphicx}
\usepackage[hyperfootnotes=false]{hyperref} 
\graphicspath{{Images/}}
\usepackage{subfig} %
\usepackage{caption} %
\usepackage{transparent}
\usepackage{float} %
\usepackage{enumitem} %
\usepackage{ mathrsfs } %
\usepackage[colorinlistoftodos]{todonotes} %
\usepackage{accents}
\usepackage{mathtools}
\usepackage{pgfplots} %
\pgfplotsset{width=6cm,compat=1.18}
\usepackage{tikz}
\usepackage{subcaption}
\usepackage{caption}
\newcommand{\de}{\, \!\mathrm d}
\renewcommand{\d}{{\!\mathrm d}}

\newcommand{\E}{\mathbb{E}}

\newcommand{\rmC}{{\ensuremath{\mathrm C}}}

\newcommand{\real}{\mathbb{R}} %
\newcommand{\Prob}{\mathbb{P}} %
\newcommand{\Nat}{\mathbb{N}} %
\newcommand{\eps}{\varepsilon} %

\newcommand{\norm}[1]{\lVert #1 \rVert} %

\newcommand{\trinorm}[1]{{\left\vert\kern-0.25ex\left\vert\kern-0.25ex\left\vert #1 
    \right\vert\kern-0.25ex\right\vert\kern-0.25ex\right\vert}} %

\newtheorem{theorem}{Theorem}[section]
\newtheorem{corollary}[theorem]{Corollary}
\newtheorem{lemma}[theorem]{Lemma}
\newtheorem{assumption}[theorem]{Assumption}
\newtheorem{definition}[theorem]{Definition}
\newtheorem{remark}[theorem]{Remark}
\newtheorem{proposition}[theorem]{Proposition}
\newtheorem{example}[theorem]{Example}

\newcommand{\Lip}{\mathop{\rm Lip}\nolimits}

\title[Linear-quadratic mean field games with Volterra dynamics]{Linear-quadratic mean field games with Volterra dynamics}

\author{Niccolò Fontana}
\address{Politecnico di Milano, Dipartimento di Matematica}
\email{niccolo.fontana@polimi.it}

\author{Mattia Martini}
\address{École Polytechnique Paris, Centre de Mathématiques Appliquées (CMAP)}
\email{mattia.martini@polytechnique.edu}

\author{Huy\^en Pham}
\address{École Polytechnique Paris, Centre de Mathématiques Appliquées (CMAP)}
\email{huyen.pham@polytechnique.edu}

\thanks{Niccolò Fontana is supported by Initiative Paul-Lévy for his research stay at École Polytechnique. 
Mattia Martini is supported by the French ANR PEPR Math-Vives project MIRTE ANR-23-EXMA-0011. N. Fontana and M. Martini were members of INdAM-GNAMPA when this work was conceived.
Huy\^en Pham is supported by the Soci\'et\'e G\'en\'erale Chair
``Risques Financiers'', the Chair AIDA-F Bank of America, 
FiME (Laboratory of Finance and Energy Markets), and the EDF-CACIB Chair ``Finance and Sustainable Development''.}
\date{}

\begin{document}

\begin{abstract}
We study a class of linear-quadratic mean field games in which the state of each agent follows a stochastic Volterra equation of convolution type. The resulting dynamics are non-Markovian and account for memory effects through a completely monotone kernel, while the mean field interaction enters both the drift and diffusion coefficients. Exploiting a Markovian lifting of the Volterra dynamics, which restores the Markov property at the price of moving to an infinite-dimensional state space, we characterize the optimal response of a representative agent for a given population mean through a system of Riccati equations in Banach spaces, thereby obtaining a feedback representation of the optimal control.
We then study the consistency condition defining the mean field equilibrium, which reduces to a forward-backward fixed point equation for the equilibrium mean: we prove existence and uniqueness of the equilibrium on short horizons and, when the control does not act on the diffusion coefficient, on arbitrary horizons under a monotonicity condition on the interaction. Finally, we consider the corresponding finite-player game and establish a quantitative propagation of chaos result. 
As a consequence, the feedback control of any mean field equilibrium, when implemented by each agent on its own state, yields an $\varepsilon_N$-Nash equilibrium for the $N$-player game, with an explicit rate as the number of agents tends to infinity.

\end{abstract}

\maketitle

\section{Introduction}
In this work, we study a class of linear-quadratic mean field games with memory, in which the state of each agent follows a stochastic Volterra equation of convolution type. 

\smallskip

Mean field games (MFGs) were introduced about twenty years ago, independently by Lasry and Lions \cite{LasryLions2006,LasryLions2006II,LasryLions2007} and by Huang, Caines, and Malhamé \cite{HuangCainesMalhame2003,HuangCainesMalhame2006,
HuangCainesMalhame2007_2}, as a framework for describing strategic interactions in stochastic systems involving a large number of agents. The key idea is to study the limiting behavior of such systems as the number of agents tends to infinity, replacing the interaction with the other players by an interaction with the aggregate behavior of the population, typically represented by the distribution of the agents' states. This leads, in the large-population limit, to a representative-agent optimization problem coupled with a consistency condition for the population distribution. Since these pioneering works, MFG theory has developed extensively from both analytical and probabilistic perspectives. We refer to \cite{CarmonaDelarue,CarmonaDelarueII} for comprehensive treatments of the theory and its probabilistic approach.

\smallskip

Among the various classes of mean field games, linear-quadratic (LQ) models have received particular attention, both for their analytical tractability and for their relevance in applications. In this setting, the state dynamics are linear and the running and terminal costs are quadratic, which often leads to explicit characterizations of equilibria in terms of Riccati equations and forward-backward systems. LQ mean field games have been studied in a variety of contexts, including decentralized equilibria \cite{HuangCainesMalhame2007}, models with common noise \cite{Graber2016}, and their connections with mean field type control \cite{BensoussanSungYamYung2016}. The LQ framework has also proved useful in applications across economics and finance, including models of systemic risk in banking systems \cite{CarmonaFouqueSun2015} and the strategic production of exhaustible resources \cite{Graber2016}. We refer to \cite[Sections~2.3.4 and~3.5]{CarmonaDelarue} and \cite[Chapter~6]{BensoussanFrehseYam2013} for further treatments of LQ mean field models.

\smallskip

The interest in stochastic Volterra equations (SVEs), dating back to the foundational work of Berger and Mizel \cite{Berger_mizel_Volterra_first, Berger_mizel_Volterra_second}, stems primarily from their versatility in capturing phenomena that classical Markovian models cannot easily reproduce such as rough sample paths and memory effects. Applications abound, from models of population dynamics to energy markets and rough volatility modeling. Yet this very versatility is the source of a fundamental mathematical challenge. Outside the special case of exponential kernels, the memory embedded in the equation destroys the Markov property of its solutions; and when the kernel is singular, i.e. $\lim_{t \to 0^+} K(t) = +\infty$, even the semimartingale property is lost. For optimal control theory, this is more than a technical inconvenience since most classical solution techniques, such as the Dynamic Programming Principle and the associated Hamilton–Jacobi–Bellman equation, are intrinsically tied to the Markov property of the underlying state process. Reconciling this requirement with genuinely path-dependent dynamics has therefore become a central theme of the recent literature, typically resolved by lifting the equation into an infinite-dimensional Banach or Hilbert space in which Markovianity is recovered by construction. Remarkable contributions in this direction include the lifting procedures developed in \cite{LQ_Volterra_AJMP, DiNunno_giordano_lifting,Gasteratos_Pannier, Hamaguchi, Possamai_Talbi_contract, Viens_Zhang_lifting}.

\smallskip

The present work is closely connected to a body of literature that has developed rapidly in recent years around optimal control and mean field games for SVEs. We cite some noteworthy contributions dealing with the LQ case. Huang, Li and Wang \cite{Huang_Li_Wang} were the first to consider a large-population LQ game whose representative state follows a stochastic Volterra equation. Their kernels are bounded, excluding singular ones, and the noise is additive, so neither the state nor the control enters the diffusion coefficient. Abi~Jaber, Neuman and Voß \cite{AJ_Neuman_Vos} extend this framework to a general class of $N$-player LQ games with $L^2$-operator costs and they treat stochastic Volterra LQ games as a particular instance of this broad class. They obtain the Nash equilibria of the $N$-player game and of its mean field limit semi-explicitly, in terms of operator resolvents, allowing for singular kernels as well as for common noise. Their approach requires the noise to enter as an exogenous input, so that neither the state nor the control acts on the diffusion coefficient, see \cite[Section~3.1 and Remark~3.2]{AJ_Neuman_Vos}. Our work departs from those of \cite{Huang_Li_Wang} and \cite{AJ_Neuman_Vos} in two respects. First, the diffusion coefficient depends on the state, on the population mean and on the control, so that the noise is multiplicative and controlled; at the same time, the kernel may be singular, which is excluded in \cite{Huang_Li_Wang}. Second, besides well-posedness under a smallness condition (which holds on short horizons), we obtain well-posedness on arbitrary horizons when the control does not act on the diffusion coefficient, under a monotonicity condition on the interaction; since the state and the population mean may still enter the diffusion coefficient, this case is covered neither by \cite{Huang_Li_Wang} nor by \cite{AJ_Neuman_Vos}.

In a different direction, Friz, Gasteratos, Horst and Theodorakopoulos \cite{Friz_Gasteratos_Horst_theodorakopoulos} study LQ mean field games driven by a common noise exploiting rough paths theory. Linear-quadratic mean field games with infinite-dimensional state spaces have also been studied recently, for Hilbert space-valued dynamics, see \cite{LiuFiroozi,FedericoGhilliGozzi}. To the best of our knowledge, the existing literature on LQ mean field games with stochastic Volterra dynamics does not cover the combination of control-dependent diffusion coefficients and possibly singular Volterra kernels considered here.

\smallskip

\textbf{Our contribution.}
We study a linear-quadratic mean field game in which the players have memory. The state of each player is governed by a stochastic Volterra equation of convolution type, so that its dynamics is genuinely non-Markovian: the whole past trajectory, weighted by a convolution kernel $K$, influences the present evolution. The mean field interaction enters through the mean of the population, and it affects both the drift and the diffusion coefficient of each player, as does the control.

Our structural assumption is that the kernel is completely monotone. By Bernstein's theorem this amounts to representing $K$ as the Laplace transform of a Radon measure $\mu$, and it is what makes the analysis tractable: following Abi Jaber, Miller and Pham \cite{LQ_Volterra_AJMP}, the Volterra dynamics can be lifted to an infinite-dimensional state, indexed by the variable $\theta$ in the support of $\mu$, which is Markovian and from which the original state is recovered by ``projection''. We choose this route because it lets the control act on the diffusion coefficient as well, a case not covered by the approaches mentioned above. 

Solving the mean field game then consists of the two usual steps, the best response and the fixed point. In the first one the population mean is frozen to a deterministic curve, and the representative player faces a linear-quadratic control problem for a Volterra equation, which is precisely the class of problems solved in \cite{LQ_Volterra_AJMP}: we obtain the optimal control in feedback form on the lifted state, characterized by a system of Riccati equations in Banach spaces. %
In the second step the consistency condition closes the game and reduces to a fixed point equation for the equilibrium mean. Since the Riccati equation associated with the frozen mean is solved backward in time, this equation has a forward-backward structure, and its solvability on arbitrary horizons is not guaranteed in general. We obtain existence and uniqueness by a contraction argument on sufficiently short horizons, under an explicit smallness condition. When the control does not act on the diffusion coefficient, we remove this restriction under a monotonicity condition on the interaction.

Finally, the interest of the mean field limit lies in what it gives back to the game with finitely many players: the equilibrium control of the limit problem provides a candidate equilibrium for the $N$-player game. 
We show that this is indeed the case for every mean field equilibrium: the strategy profile in which every player implements the equilibrium feedback on its own state is an $\eps_N$-Nash equilibrium, with an explicit rate $\eps_N$ of order $N^{-1/2}$. The proof rests on a propagation of chaos for the lifted particles, which in particular yields the corresponding convergence for the states themselves. The results available for Volterra mean field dynamics at the level of the states, such as \cite{Huang_Li_Wang}, are not sufficient for our purposes, since the equilibrium control is given in feedback form in terms of the lifted state, and the estimates have therefore to be established at that level.

Although we do not develop them here, these results open the way to applications in which both memory and strategic interaction are present. In advertising models the goodwill of a firm reacts to past investment through a decaying kernel, and firms competing for the same market are naturally described as players in a game; in energy markets, production and storage decisions are constrained by inertia and by the persistence of past choices, while agents interact through the aggregate supply. In both cases the linear-quadratic structure keeps the equilibrium explicitly computable, and the mean field formulation is the natural device for a large population of symmetric agents. We leave the study of such models to future work.

\medskip

\textbf{Organization of the paper.} 
In Section~\ref{sec:formulation} we introduce the Volterra $N$-player game and its mean field counterpart, and state the two main results: the well-posedness of the mean field game (Theorem~\ref{thm:mfg-wellposed}) and the approximate Nash property (Theorem~\ref{thm:eps-nash}). Section~\ref{sec:mfg-resolution} proves the former, following the two classical steps of the mean field game paradigm: %
 we first solve the best response of a representative player for a frozen mean, and we then close the fixed point induced by the consistency condition on short horizons; we finally extend the result to arbitrary horizons when the control does not act on the diffusion coefficient.
Section~\ref{sec:propagation} proves the latter: we first establish a propagation of chaos, quantifying how each player in the $N$-player game approaches its mean field limit, and we then deduce the approximate Nash property. Some technical material is deferred to the appendices: Appendix~\ref{appendix:volterra tools} collects classical results on convolution Volterra kernels, Appendix~\ref{appendix:B} presents the estimates on the Riccati equations underlying the mean field game, and Appendix~\ref{appendix:convergence tools} presents the auxiliary results used in the propagation of chaos.

\section{Problem formulation and main results}\label{sec:formulation}
We are interested in a linear-quadratic stochastic differential $N$-player game of Volterra type: the state of each player follows a controlled dynamics that is linear in the state and in the control and depends also on the past through a suitable convolution kernel, while the cost is quadratic.

For the rest of the discussion, we fix a complete probability space $(\Omega,\mathcal{F},\Prob)$ supporting a family of independent one-dimensional Brownian motions $\{W^i\}_{i\geq 0}$, and we denote by $\mathbb{F}=(\mathcal{F}_t)_{t \in [0,T]}$ the $\Prob$-augmentation of the filtration generated by $\{W^i\}_{i\geq0}$, which satisfies the usual hypotheses. For simplicity we denote $W^0$ by $W$.

\medskip

 We fix a finite time horizon $T>0$ and two integers $d,q \geq 1$, standing respectively for the dimensions of the state and of the control. The state dynamics are governed by a continuous initial curve $g_0 \colon [0,T] \to \real^d$, two measurable functions $\beta,\gamma \colon [0,T] \to \real^d$, the matrices $B,D \in \real^{d\times d}$ and $C,F \in \real^{d\times q}$, and a scalar convolution kernel $K \colon [0,T] \to \real$. Throughout, $|\cdot|$ denotes the Euclidean norm on vectors and the associated operator norm on matrices, so that $|Ax| \leq |A|\,|x|$ and $|AB| \leq |A|\,|B|$ whenever the products are defined. The set of admissible controls is
\[
\mathcal{A} := \Big\{ \alpha \colon \Omega \times [0,T] \to \real^q \ \ \mathbb{F}\text{-progressively measurable} \ : \ \sup_{t\in[0,T]}\E\big[|\alpha_t|^4\big] < \infty \Big\}.
\]
The running cost is the quadratic function $f \colon \real^d \times \real^q \to \real$,
\begin{equation}
\label{eq:running-cost}
    f(x,a) := x^\top Q x + a^\top R a + 2 x^\top L,
\end{equation}
where $Q \in S^d_+$, $R \in S^q_+$ and $L \in \real^d$; here $S^k_+$ denotes the cone of symmetric positive semi-definite $k \times k$ matrices. The structural and regularity conditions on the data used throughout the paper are gathered in Assumption~\ref{ass:kernel} and Assumption~\ref{ass:data} below.

\subsection{Volterra \texorpdfstring{$N$}{N}-player game}
\label{subsec:Nplayer}
We consider a symmetric game among $N \geq 2$ players. The state $X^{i,N} \in \real^d$ of player $i \in \{1,\dots,N\}$ evolves according to a controlled stochastic Volterra integral equation in which the interaction enters through the empirical mean of the states. More precisely, given an admissible profile $\boldsymbol{\alpha} = (\alpha^{1,N},\dots,\alpha^{N,N}) \in \mathcal{A}^N$, the states solve, for every $1 \leq i \leq N$,
\begin{equation}
\label{eq:Nplayer-SVE}
\begin{aligned}
    X_t^{i,N} &= g_0(t) + \int_0^t K(t-s)\Big[\beta(s) + B\big(X_s^{i,N} - \overline{X}^N_s\big) + C\alpha_s^{i,N}\Big]\de s\\
    &\quad + \int_0^t K(t-s)\Big[\gamma(s) + D\big(X_s^{i,N} - \overline{X}^N_s\big) + F\alpha_s^{i,N}\Big]\de W_s^i, \qquad t \in [0,T],
\end{aligned}
\end{equation}
where
\[
\overline{X}^N_t := \frac{1}{N}\sum_{j=1}^N X^{j,N}_t
\]
is the empirical mean of the population, and each player is driven by its own independent Brownian motion $W^i$. Here $\boldsymbol{\alpha} = (\alpha^{1,N},\dots,\alpha^{N,N})$ ranges over the set of admissible profiles $\mathcal{A}^N$, that is, each component $\alpha^{i,N}$ belongs to $\mathcal{A}$; equivalently, $\boldsymbol{\alpha}$ is an $\mathbb{F}$-progressively measurable process with values in $\real^{Nq}$ such that $\sup_{t\in[0,T]}\E[|\boldsymbol{\alpha}_t|^4] < \infty$. 
Since $\mathbb{F}$ is generated by the whole family $\{W^i\}_{i\geq0}$, an admissible control may depend on the noise of the entire system: this corresponds to a centralized information structure. We work with this larger class because it enlarges the set of deviations against which the inequality of Definition~\ref{def:nash} has to hold, making the approximate Nash property of Theorem~\ref{thm:eps-nash} stronger. However, the strategies we construct in Section~\ref{sec:propagation} are of feedback type, each player using only its own state, in the sense made precise in Section~\ref{sec:standing-assumptions}. Feedback strategies fit into this framework: once the corresponding closed loop system is solved, a feedback profile defines a process, which belongs to $\mathcal{A}^N$ as soon as it has bounded fourth moments.

 Under Assumptions~\ref{ass:kernel} and~\ref{ass:data}, for every $N \geq 2$ and every admissible profile $\boldsymbol{\alpha} \in \mathcal{A}^N$ the coupled system \eqref{eq:Nplayer-SVE} admits a unique continuous adapted solution with $\sup_{1\leq i \leq N}\sup_{t\in[0,T]}\E[|X^{i,N}_t|^4] < \infty$; the well-posedness is discussed in Section~\ref{sec:propagation}.

Each player minimizes an individual cost that couples its own state and control with the empirical mean of the population. The cost of player $i$ associated with the profile $\boldsymbol{\alpha} \in \mathcal{A}^N$ is
\begin{equation}
\label{eq:Nplayer-cost}
    J^{i,N}(\boldsymbol{\alpha}) := \E\left[ \int_0^T \Big\{ f\big(X_t^{i,N},\alpha_t^{i,N}\big) + \big(\overline{X}^N_t\big)^\top Q\, \overline{X}^N_t + 2\big(\overline{X}^N_t\big)^\top L \Big\}\de t \right],
\end{equation}
where $f$ is the running cost \eqref{eq:running-cost}.

For a profile $\boldsymbol{\alpha} \in \mathcal{A}^N$, a player index $i$ and an admissible control $\eta \in \mathcal{A}$, we write
\[
(\boldsymbol{\alpha}^{-i},\eta) := (\alpha^{1,N},\dots,\alpha^{i-1,N},\eta,\alpha^{i+1,N},\dots,\alpha^{N,N})
\]
for the profile in which player $i$ deviates to $\eta$ while the other players keep their strategies. In general, the players $j\neq i$ keep the processes $\alpha^{j,N}$ (open loop convention). For profiles in feedback form, such as the one constructed in Section~\ref{sec:propagation}, the players $j\neq i$ keep instead their feedback maps, so that their controls in $(\boldsymbol{\alpha}^{-i},\eta)$ are obtained by applying these maps to their states in the new system, and in general differ from $\alpha^{j,N}$; the precise convention is given in Section~\ref{sec:standing-assumptions}.

The solution concept we adopt for the $N$-player game is that of Nash equilibrium. Since constructing exact equilibria for a finite number of players is in general delicate, we also introduce a relaxed notion, the $\eps$-Nash equilibrium.
\begin{definition}[Nash and approximate Nash equilibria]
\label{def:nash}
    Let $\eps \geq 0$. A profile $\boldsymbol{\alpha}^\star \in \mathcal{A}^N$ is an \emph{$\eps$-Nash equilibrium} of the $N$-player game if, for every player $1 \leq i \leq N$ and every admissible deviation $\eta \in \mathcal{A}$,
    \begin{equation}
        J^{i,N}(\boldsymbol{\alpha}^\star) \leq J^{i,N}\big(\boldsymbol{\alpha}^{\star,-i},\eta\big) + \eps.
    \end{equation}
    When $\eps = 0$, $\boldsymbol{\alpha}^\star$ is a \emph{Nash equilibrium}.
\end{definition}

\subsection{Volterra mean field game}
\label{subsec:MFG-model}
Explicitly identifying Nash equilibria of the $N$-player game of Section~\ref{subsec:Nplayer} is in general difficult, the more so as the number of players grows. The mean field game is the asymptotic device that circumvents this difficulty by exploiting the symmetry of the players. Heuristically, as $N \to \infty$ the empirical mean $\overline{X}^N$ is expected to concentrate, by a law of large numbers, around a deterministic curve, which plays the role of the population mean. Freezing the interaction term to such a curve decouples the players and reduces the analysis to the control problem faced by a single \emph{representative player}, to be closed by the self-consistency requirement that the frozen curve coincide with the mean of the representative optimal state.

Given a deterministic curve $m$ in the space $\rmC([0,T];\real^d)$ of continuous $\real^d$-valued functions on $[0,T]$, the representative player controls the stochastic Volterra integral equation
\begin{equation}
\label{eq:MFG-SVE}
\begin{aligned}
    X_t^{\alpha} &= g_0(t) + \int_0^t K(t-s)\big[\beta(s) + B(X_s^{\alpha} - m_s) + C\alpha_s\big]\de s\\
    &\quad + \int_0^t K(t-s)\big[\gamma(s) + D(X_s^{\alpha} - m_s) + F\alpha_s\big]\de W_s, \qquad t \in [0,T],
\end{aligned}
\end{equation}
driven by $W = W^0$, and minimizes over $\alpha \in \mathcal{A}$ the cost
\begin{equation}
\label{eq:MFG-cost}
    J^m(\alpha) := \E\left[ \int_0^T \Big\{ f(X_t^{\alpha},\alpha_t) + m_t^\top Q\, m_t + 2 m_t^\top L \Big\}\de t \right].
\end{equation}
Since $m$ is continuous, hence bounded, equation \eqref{eq:MFG-SVE} is well-posed for every $\alpha \in \mathcal{A}$, and its solution satisfies $\sup_{t\in[0,T]}\E[|X^\alpha_t|^4] < \infty$ (see \cite[Theorem 4.4]{LQ_Volterra_AJMP}).

\begin{remark}
\label{rmk:cost-coupling}
    The population mean influences a player in two structurally different ways. In the dynamics \eqref{eq:Nplayer-SVE} it is a genuine coupling, entering the drift and diffusion coefficients through $B$ and $D$. In the cost \eqref{eq:Nplayer-cost}, instead, it appears only through the additive term $(\overline{X}^N)^\top Q\,\overline{X}^N + 2(\overline{X}^N)^\top L$, with no cross term coupling $X^{i,N}$ with $\overline{X}^N$. Accordingly, once the mean is frozen to a curve $m$, the corresponding contribution $m_t^\top Q\, m_t + 2 m_t^\top L$ in \eqref{eq:MFG-cost} is deterministic and does not affect the minimization, so that the representative player faces a standard linear-quadratic Volterra control problem. In particular, the matrices $Q,L$ in the interaction term could be replaced by arbitrary $\bar{Q} \in S^d_+$ and $\bar{L} \in \real^d$ without altering any of the arguments; we keep $Q,L$ only to lighten the notation.
\end{remark}

The consistency condition then closes the game: it selects, among all frozen curves $m$, the one that coincides with the mean of the corresponding optimal state, which we denote by $\mathfrak{m}$. This leads to the following notion of solution.

\begin{definition}[Mean field game solution]
\label{def:MFG-solution}
    A \emph{solution} of the Volterra mean field game is a pair $(\hat{\alpha},\mathfrak{m}) \in \mathcal{A} \times \rmC([0,T];\real^d)$ such that:
    \begin{enumerate}
        \item[(i)] \emph{(best response)} $\hat{\alpha}$ minimizes $J^{\mathfrak{m}}$ over $\mathcal{A}$, where $X^{\hat\alpha}$ is the solution of \eqref{eq:MFG-SVE} associated with $\mathfrak{m}$;
        \item[(ii)] \emph{(consistency)} the following fixed point condition holds:
        $$
        \mathfrak{m}_t = \E\big[X_t^{\hat\alpha}\big], \qquad t \in [0,T].
        $$
    \end{enumerate}
\end{definition}

\subsection{Assumptions and main results}
\label{sec:standing-assumptions}
In this section we collect the standing assumptions on the data and state the two main results of the paper, whose proofs occupy Sections~\ref{sec:mfg-resolution} and~\ref{sec:propagation}.

\begin{assumption}[Kernel]
\label{ass:kernel}
The scalar kernel $K \colon (0,T] \to \real$ satisfies the following conditions.
\begin{enumerate}
    \item[(i)] $K$ is completely monotone, that is, $(-1)^k\frac{\d^k}{\d t^k}K(t) \geq 0$ for every $t > 0$ and every $k \in \Nat$. By Bernstein's theorem \cite[Theorem 2.5]{Gripenberg-Londed-Staffans}, there exists a Radon measure $\mu$ on $[0,+\infty)$ such that
    $$
    K(t) = \int_0^\infty e^{-t\theta}\,\mu(\de\theta).
    $$
    \item[(ii)] There exists $\zeta \in (0,\tfrac{1}{4})$ such that
    \begin{equation}
    \label{hp:integrability on mu}
        \int_0^\infty 1\land \theta^{-\zeta}\,\mu(\de\theta) < \infty.
    \end{equation}
\end{enumerate}
\end{assumption}
Condition~(ii) in Assumption~\ref{ass:kernel} ensures that the kernel is at least square-integrable, which gives the well-posedness of the stochastic integral in the It\^o sense, and that the trajectories of $X^\alpha$ are continuous (see \cite{LQ_Volterra_AJMP} and \cite{Hamaguchi} for further details). We also highlight that the class of completely monotone kernels is sufficiently broad to include many kernels of practical interest, including singular ones.  
\begin{example} We give some notable examples of completely monotone kernels satisfying Assumption~\ref{ass:kernel}.
    \begin{itemize}
        \item[(i)]  The exponential kernel: for $\theta_0 \geq 0$
        $$
        K(t) = e^{-\theta_0t}, \quad t \geq 0
        $$
        has as corresponding Bernstein measure $\mu(\de \theta) = \delta_{\theta_0}(\de\theta)$. The case $\theta_0=0$ corresponds to $K\equiv1$, that is, to the classical Markovian setting.
        \item[(ii)] The fractional kernel: for $H \in (\frac14,\frac12)$ 
        $$
        K(t) = \frac{1}{\Gamma(H+1/2)}t^{H-1/2}, \quad t >0
        $$
        has as corresponding Bernstein measure $\mu(\de\theta)=\frac{\theta^{-H-1/2}}{\Gamma(H+1/2)\Gamma(1/2-H)}\de\theta$.
        \item[(iii)] The gamma kernel: for $\theta_0 > 0$ and $H \in (\frac14,\frac12)$ 
        $$
        K(t) = \frac{1}{\Gamma(H+1/2)}t^{H-1/2}e^{-\theta_0t}, \quad t >0
        $$
        has as corresponding Bernstein measure $\mu(\de\theta)=\frac{(\theta-\theta_0)^{-H-1/2}}{\Gamma(H+1/2)\Gamma(1/2-H)}\mathbf{1}_{(\theta_0,\infty)}(\theta)\de\theta$.
        \item[(iv)] In general, if $K_1$ and $K_2$ are two completely monotone kernels, then $K_1+K_2$ and $K_1K_2$ are completely monotone as well. If $K_1$ and $K_2$ satisfy Assumption~\ref{ass:kernel}, so does $K_1+K_2$, and so does $K_1K_2$ as soon as one of the two kernels is bounded; for instance, (iii) is the product of (ii) and (i).
    \end{itemize}
\end{example}
\begin{assumption}[Coefficients]
\label{ass:data}
The coefficients of the problem satisfy the following conditions.
\begin{enumerate}
    \item[(i)] The initial curve $g_0 \colon [0,T] \to \real^d$ is continuous, hence bounded on $[0,T]$.
    \item[(ii)] The functions $\beta,\gamma \colon [0,T] \to \real^d$ are measurable and bounded.
    \item[(iii)] The control cost is coercive: there exists $\lambda > 0$ such that $R - \lambda I_q \in S^q_+$. In particular $R$ is positive definite, which strengthens the standing requirement $R \in S^q_+$ from the definition of the running cost \eqref{eq:running-cost}.
\end{enumerate}
\end{assumption}

We can now state the two main results of the paper. The first one concerns the resolution of the mean field game.

\begin{restatable}[Well-posedness of the mean field game]{theorem}{mfg}
\label{thm:mfg-wellposed}
    Suppose Assumptions~\ref{ass:kernel} and~\ref{ass:data} hold, and set
    $$
    \rho_T := \Lip_\psi\,\norm{K}_{L^1(0,T)},
    $$
    where $\Lip_\psi$ is the constant given in Lemma~\ref{lemma:psi_lip}, which depends only on $\mu$, $T$ and the matrices $B,C,D,F,Q,R$. If $\rho_T<1$, the Volterra mean field game of Definition~\ref{def:MFG-solution} admits a unique solution $(\hat{\alpha},\mathfrak{m}) \in \mathcal{A} \times \rmC([0,T];\real^d)$.
\end{restatable}
The proof is carried out in Section~\ref{sec:mfg-resolution} and is based on a Markovian lifting procedure, following the approach of \cite{LQ_Volterra_AJMP}: exploiting the specific structure of the Volterra kernel, the non-Markovian dynamics can be reformulated as an infinite-dimensional Markovian system. Within this setting, for a frozen curve $m$, the optimal control of the representative player is obtained in affine feedback form, characterized by a system of Riccati equations. The consistency condition then reduces to a fixed point equation for the mean of the lifted state, driven by a map $\psi$ which is Lipschitz with constant $\Lip_\psi$; the quantity $\rho_T$ is the resulting contraction constant on $[0,T]$, and the condition $\rho_T<1$ holds on every sufficiently short horizon (Remark~\ref{rmk:rho}).
When the control does not act on the diffusion coefficient and the interaction pulls each player towards the population mean, the smallness condition can be removed, see Theorem~\ref{thm:global F=0}; some condition of this type is needed in general, see Remark~\ref{rmk:global F=0}.

\medskip

The second result concerns the $N$-player game, where we consider profiles in \emph{distributed feedback form}: each player applies the same map $\phi(t,\cdot)$ to the Markovian lift $Y^{i,N}_t\in L^1_\mu$ of its own state, see \eqref{eq:particles lifted}. For such profiles, the deviation in Definition~\ref{def:nash} is understood in the closed loop sense of \cite[Section~2.1.2]{CarmonaDelarue}: the players $j\neq i$ keep the map $\phi$, applied to their own lifted states, while player $i$ may use any $\eta\in\mathcal{A}$; the open loop convention is discussed in Remark~\ref{rmk:open loop}. By Remark~\ref{rmk:feedback}, every solution of the mean field game provides such a feedback, and the second result shows that it yields an approximate Nash equilibrium, with an explicit rate in the number of players.
\begin{restatable}[Approximate Nash equilibrium]{theorem}{epsnash}
\label{thm:eps-nash}
    Suppose Assumptions~\ref{ass:kernel} and~\ref{ass:data} hold, and let $(\hat\alpha,\mathfrak{m})$ be a solution of the mean field game in the sense of Definition~\ref{def:MFG-solution}. For every $N \geq 2$, let $(\alpha^{i,N})_{1\leq i\leq N}$ be the strategy profile in which each player applies the feedback control of $\hat\alpha$ to its own lifted state. Then there exists a constant $c>0$, independent of $N$, such that $(\alpha^{i,N})_{1\leq i\leq N}$ is an $\eps_N$-Nash equilibrium of the $N$-player game, in the sense of Definition~\ref{def:nash}, with $\eps_N \leq c/\sqrt{N}$ (and so $\eps_N \downarrow 0$ as $N\to\infty$).
\end{restatable}
Theorem~\ref{thm:eps-nash} only requires the existence of a solution of the mean field game, guaranteed for instance by Theorem~\ref{thm:mfg-wellposed} or Theorem~\ref{thm:global F=0}, and not its uniqueness.
A key ingredient in the proof is a propagation of chaos result, which quantitatively couples the state of each player in the $N$-player game with its mean field counterpart, and shows that their distance vanishes as $N \to \infty$. This analysis, of independent interest, is developed in Subsection~\ref{ssec:convergence} (see in particular Proposition~\ref{prop:propagation chaos}).

\bigskip

\textbf{On constants. }Throughout the paper, $c$ denotes a strictly positive constant that depends only on the \emph{data of the problem}, namely the horizon $T$, the measure $\mu$, the coefficients of the dynamics and of the running cost, the sup-norms $\norm{g_0}_\infty, \norm{\beta}_\infty, \norm{\gamma}_\infty$, and the coercivity constant $\lambda$ of Assumption~\ref{ass:data}(iii). Within proofs, the value of $c$ may change from line to line. Any further dependence, such as on a moment exponent $p$, is indicated by writing $c=c(p)$, and dependence or independence with respect to the number of players $N$ is always stated explicitly.  In Section~\ref{sec:propagation}, where a solution $(\hat\alpha,\mathfrak{m})$ of the mean field game is fixed, $c$ may also depend on it.

\section{Mean field game resolution}
\label{sec:mfg-resolution}
By Definition~\ref{def:MFG-solution}, solving the mean field game amounts to finding a pair $(\hat{\alpha},\mathfrak{m})$ that satisfies the best response condition and the consistency condition. We address the two in turn. In Section~\ref{subsec:best response} we freeze the curve $m$ and solve the resulting linear-quadratic Volterra control problem faced by the representative player: relying on the Markovian lifting of the Volterra dynamics, we characterize its optimal control in feedback form through a system of Riccati equations. In Section~\ref{subsec:fixed point} we then impose the consistency condition, which reduces to a fixed point equation for the equilibrium mean $\mathfrak{m}$; we prove that this equation admits a unique solution by a contraction argument, under a suitable smallness condition. Together, the two steps yield the well-posedness stated in Theorem~\ref{thm:mfg-wellposed}. Finally, in Section~\ref{subsec:arbitrary horizons} we extend this result to arbitrary horizons when the control does not act on the diffusion coefficient.

\subsection{Best response}
\label{subsec:best response}

We now address the best response, that is, condition~(i) of Definition~\ref{def:MFG-solution}. We fix a curve $m \in \rmC([0,T];\real^d)$ and minimize the cost $J^m$ in \eqref{eq:MFG-cost} over $\alpha \in \mathcal{A}$, subject to the controlled dynamics \eqref{eq:MFG-SVE}. As observed in Remark~\ref{rmk:cost-coupling}, since $m$ is deterministic the additive contribution $m_t^\top Q m_t + 2 m_t^\top L$ does not affect the minimization, so that the representative player faces a linear-quadratic control problem for a stochastic Volterra equation of convolution type. This is precisely the class of problems solved by Abi Jaber, Miller and Pham \cite{LQ_Volterra_AJMP}, whose results we specialize to the present setting.

\smallskip

The key idea of their approach is a Markovian lifting of the Volterra dynamics. Writing $\mu[\varphi] := \int_0^\infty \varphi(\theta)\,\mu(\de\theta)$, one introduces the lifted state $Y^\alpha_t = (Y^\alpha_t(\theta))_{\theta > 0}$, related to the original state by the \emph{projection} $X^\alpha_t = g_0(t) + \mu[Y^\alpha_t]$, which solves
\begin{align*}
    Y_t^\alpha(\theta) &= \int_0^t e^{-\theta(t-s)}[\beta(s)+Bg_0(s)+B(\mu[Y_s^\alpha]-m_s)+C\alpha_s]\de s\\
    &\qquad+ \int_0^t e^{-\theta(t-s)}[\gamma(s) +Dg_0(s) + D(\mu[Y_s^\alpha]-m_s) + F\alpha_s]\de W_s.
\end{align*}
By \cite[Theorems~2.3 and~4.4]{LQ_Volterra_AJMP}, for every $\alpha\in\mathcal{A}$ the lifted state $Y^\alpha$ is well defined, has continuous sample paths in $L^1_\mu$, and satisfies
\begin{equation}
\label{eq:lifted moments}
    \sup_{t\in[0,T]}\E\big[\norm{Y^\alpha_t}_{L^1_\mu}^4\big]<\infty .
\end{equation}
The resolution of this control problem is governed by a system of Riccati equations. For $\theta,\tau \in (0,+\infty)$, consider
\begin{equation}
\label{eq:Riccati}
    \begin{dcases}
    \dot{\Gamma}_t(\theta,\tau) &= (\theta+\tau)\Gamma_t(\theta,\tau) - \mathcal{R}_1(\Gamma_t)(\theta,\tau), \quad \Gamma_T(\theta,\tau)=0,\\
    \dot{\Lambda}_t^m(\theta) &= \theta\Lambda_t^m(\theta) - \mathcal{R}_2(t,\Gamma_t,\Lambda_t^m,m)(\theta), \quad \Lambda^m_T(\theta)=0.
    \end{dcases}
\end{equation}
The right-hand sides of \eqref{eq:Riccati} are built from three auxiliary operators, which will also appear in the optimal feedback: a cross term $S$ coupling state and control, the effective control cost $\hat{R}$, and an affine term $h$ collecting the contributions of the exogenous coefficients. For every $G \in L^1(\mu \otimes \mu; \real^{d\times d})=:L^1_{\mu\otimes\mu}$ and $\ell \in L^1_\mu$, with the notation $\mu^{\otimes 2}[G]:= \int_0^\infty \int_0^\infty G(\theta,\tau)\mu(\de\theta)\mu(d\tau)$, we set
\begin{equation}
    \begin{aligned}
        S(G)(\tau) & := C^\top \mu[G(\cdot,\tau)] + F^\top \mu^{\otimes 2}[G] D, \\[0.2cm]
        \hat{R}(G) &:= R + F^\top \mu^{\otimes 2}[G]F, \\[0.2cm]
        h(t,G,\ell,m) & := C^\top \mu[\ell] + F^\top \mu^{\otimes 2}[G] (\gamma(t)+Dg_0(t)-Dm_t).
    \end{aligned}
\end{equation}
In terms of these, the nonlinear terms of the two Riccati equations are given by
\begin{equation}
    \begin{aligned}
        \mathcal{R}_1(G)(\theta,\tau) & := Q + D^\top \mu^{\otimes 2}[G]D + B^\top \mu[G(\cdot,\tau)] + \mu[G(\theta,\cdot)]B  - S(G)(\theta)^\top \hat{R}(G)^{-1} S(G)(\tau),
    \end{aligned}
\end{equation}
\begin{equation}
    \begin{aligned}
        \mathcal{R}_2(t,G,\ell,m)(\theta) & := L + Qg_0(t) + B^\top \mu[\ell] +  \mu[G(\theta,\cdot)](\beta(t)+Bg_0(t)-Bm_t) \\
        & \qquad + D^\top \mu^{\otimes 2}[G](\gamma(t)+Dg_0(t)-Dm_t)  - S(G)(\theta)^\top \hat{R}(G)^{-1}h(t,G,\ell,m).
    \end{aligned}
\end{equation}

To give a precise meaning to the solutions to the Riccati equations \eqref{eq:Riccati}, we give the following definitions:
\begin{definition}\label{def:gamma}
    Let $G \in L^\infty(\mu \otimes \mu; \real^{d\times d})=: L^\infty_{\mu\otimes\mu}$. We say that $G$ is symmetric if
    $$
    G(\theta,\tau) = G(\tau,\theta)^\top, \quad \mu\otimes\mu\text{-a.e.}
    $$
    and non-negative if
    $$
    \int_{\real^2_+} \phi(\theta)^\top G(\theta,\tau)\phi(\tau)\mu(\de\theta)\mu(d\tau) \geq 0, \quad \forall \phi \in L^1_\mu.
    $$
    We denote by $S^d_+(\mu\otimes \mu)$ the set of all symmetric non-negative $G \in L^\infty_{\mu\otimes\mu}$.
\end{definition}

\begin{definition}\label{def:sol_riccati}
    For a fixed $m \in \rmC([0,T];\real^d)$, a solution to \eqref{eq:Riccati} is a couple $(\Gamma,\Lambda^m) \in \rmC([0,T];L^1_{\mu \otimes\mu})\times \rmC([0,T];L^1_\mu)$ such that
    \begin{align}
        \Gamma_t(\theta,\tau) & = \int_t^T e^{-(\theta+\tau)(s-t)}\mathcal{R}_1(\Gamma_s)(\theta,\tau)\de s, \quad t \in [0,T], \mu\otimes\mu\text{-a.e.} \\
        \Lambda_t^m(\theta) & = \int_t^T e^{-\theta(s-t)}\mathcal{R}_2(s,\Gamma_s,\Lambda^m_s,m)(\theta)\de s, \quad t \in [0,T], \mu\text{-a.e.}
    \end{align}
    In particular, $\hat{R}(\Gamma_t)$ is invertible for every $t \in [0,T]$.
\end{definition}
The next proposition is a direct consequence of \cite[Theorem 3.4]{LQ_Volterra_AJMP}.
\begin{proposition}\label{prop:gamma}
    Suppose Assumption~\ref{ass:kernel} and Assumption~\ref{ass:data} hold. Then, for every $m \in \rmC([0,T];\real^d)$ there exists a unique couple $(\Gamma,\Lambda^m)$ solution in the sense of Definition \ref{def:sol_riccati} to \eqref{eq:Riccati}. Moreover, $\Gamma_t \in S^d_+(\mu\otimes\mu)$ for every $t \in [0,T]$ and there exists a positive constant $c_\Gamma > 0$ such that
    $$
    \int_0^\infty |\Gamma_t(\theta,\tau)|\mu(d\tau) \leq c_\Gamma, \quad t \in [0,T], \mu\text{-a.e.}
    $$
\end{proposition}
\begin{remark}
    Notice that the unique solution $(\Gamma,\Lambda^m)$ to \eqref{eq:Riccati} is such that for every $t \in [0,T]$, $\Gamma_t \in L^1_{\mu\otimes\mu} \cap L^\infty_{\mu\otimes\mu}$ and $\Lambda^m_t \in L^1_\mu\cap L^\infty_\mu$. Moreover, notice that the solution of the first Riccati equation $\Gamma_t$ does not depend on $m$, thus it is the same solution as in the non-mean field case.
\end{remark}

We can now state the resolution of the best response problem, which specializes the verification theorem of \cite[Theorem~3.7]{LQ_Volterra_AJMP} to the dynamics \eqref{eq:MFG-SVE}.
\begin{proposition}[Best response]
\label{prop:best-response}
    Suppose Assumptions~\ref{ass:kernel} and~\ref{ass:data} hold. Let $m \in \rmC([0,T];\real^d)$ and let $(\Gamma,\Lambda^m)$ be the unique solution of the Riccati system \eqref{eq:Riccati} provided by Proposition~\ref{prop:gamma}. Then the problem of minimizing $J^m$ over $\mathcal{A}$ admits a unique optimal control $\alpha^*$, given in feedback form in terms of the lifted state by
    \begin{equation}
    \label{eq:LQ-MFG feedback control}
        \alpha^*_t = -\hat{R}(\Gamma_t)^{-1} \left\{ h(t,\Gamma_t,\Lambda_t^m,m) + \int_0^\infty S(\Gamma_t)(\theta)Y^{\alpha^*}_t(\theta)\mu(\de\theta) \right\}.
    \end{equation}
\end{proposition}
\begin{proof}
    The fact that $\alpha^*$ defined by \eqref{eq:LQ-MFG feedback control} is an optimal control follows by a direct application of \cite[Theorem~3.7]{LQ_Volterra_AJMP}. Regarding the uniqueness, one can straightforwardly verify that the cost functional $J^m:\mathcal{A} \to \real$ is strictly convex thanks to Assumption \ref{ass:data}(iii), namely for every $\vartheta \in (0,1)$ and $\alpha^1,\alpha^2 \in \mathcal{A}$ differing on a set of positive $\de t\otimes\de\Prob$-measure, it holds $J^m(\vartheta\alpha^1+(1-\vartheta)\alpha^2) < \vartheta J^m(\alpha^1)+(1-\vartheta)J^m(\alpha^2)$. Hence the minimizer $\alpha^*$ of $J^m$ obtained through \cite[Theorem~3.7]{LQ_Volterra_AJMP} is unique, up to $\de t\otimes\de\Prob$-null sets.
\end{proof}

\subsection{Fixed point}
\label{subsec:fixed point}
We can now turn to the consistency condition~(ii) of Definition~\ref{def:MFG-solution}, which closes the mean field game by requiring the frozen curve to coincide with the mean of the optimally controlled state. Fix $m \in \rmC([0,T];\real^d)$ and let $\alpha^*$ be the corresponding best response of Proposition~\ref{prop:best-response}, with optimally controlled lifted trajectory $Y^* := Y^{\alpha^*}$. Both $\alpha^*$ and $Y^*$ clearly depend on the curve $m$, through the Riccati solution $\Lambda^m$ and the terms involving $m$ in the dynamics; to keep the notation light, we do not display this dependence explicitly. Imposing $\E[X^{\alpha^*}_t] = m_t$ for every $t \in [0,T]$ and recalling that $X^{\alpha^*}_t = g_0(t) + \mu[Y^*_t]$, we substitute $\alpha^*$ into the lifted dynamics and take expectations. Since the stochastic integral has zero expectation, and since the bound \eqref{eq:lifted moments} allows us to exchange expectation and integrals, this yields
\begin{equation}\label{eq:fixed_point_frozen}
\begin{aligned}
    \E[Y^*_t(\theta) ] &= \int_0^t e^{-\theta(t-s)}\bigg\{ \beta(s) + Bg_0(s) - Bm_s + B\E[\mu[Y_s^*] ] \\
    &\qquad - C\hat{R}(\Gamma_s)^{-1}\left[ h(s,\Gamma_s,\Lambda^{m}_s,m)+\int_0^\infty S(\Gamma_s)(\tilde{\theta})\E[Y_s^*(\tilde{\theta}) ]\mu(d\tilde{\theta}) \right] \bigg\}ds.
\end{aligned}
\end{equation}

We look for a solution of the form $m = g_0 + \mu[M]$, namely $m_t = g_0(t)+\mu[M_t]$, for some $M \in \rmC([0,T];L^1_\mu)$; such a curve is continuous, since $|\mu[M_t]-\mu[M_s]| \leq \norm{M_t-M_s}_{L^1_\mu}$. For $M \in \rmC([0,T];L^1_\mu)$ we write $m^M := g_0+\mu[M]$ and $\norm{M}_\infty := \sup_{t\in[0,T]}\norm{M_t}_{L^1_\mu}$. Since the second Riccati equation in \eqref{eq:Riccati} is solved backward in time, $\Lambda^{m}_t$ depends on the whole restriction $(m_s)_{s\in[t,T]}$, and not only on $m_t$. Accordingly, the operator describing the mean of the optimal control acts on paths: for $M \in \rmC([0,T];L^1_\mu)$ and $t\in[0,T]$ we set
\begin{equation}
    \label{eq:definition of psi}
    \psi_t(M) := -C\hat{R}(\Gamma_t)^{-1}\left\{ h\big(t,\Gamma_t,\Lambda^{m^M}_t,m^M\big)+\int_0^\infty S(\Gamma_t)(\theta)M_t(\theta)\,\mu(\de\theta) \right\},
\end{equation}
where $\Lambda^{m^M}$ is the solution of the second Riccati equation in \eqref{eq:Riccati} associated with the curve $m^M$. Replacing $m$ by $m^M$ and $\E[Y^*]$ by $M$ in \eqref{eq:fixed_point_frozen}, the terms $B(g_0(s)-m^M_s)+B\mu[M_s]$ cancel, and we are led to the fixed point equation
\begin{equation}
\label{eq:ODE fixed point}
    M_t(\theta) = \int_0^t e^{-\theta(t-s)}\left\{ \beta(s) + \psi_s(M)\right\}\de s, \qquad t \in [0,T],\ \mu\text{-a.e.}
\end{equation}
Since $\psi_s(M)$ depends on the restriction of $M$ to $[s,T]$, \eqref{eq:ODE fixed point} is not a Volterra equation of forward type: it couples the forward evolution of $M$ with the backward evolution of $\Lambda^{m^M}$. This derivation is only heuristic, since it replaces the unknown mean $\E[Y^*]$ by $M$. The next lemma makes it rigorous in both directions: solutions of \eqref{eq:ODE fixed point} produce equilibria, and every equilibrium produces a solution of \eqref{eq:ODE fixed point}. 
\begin{lemma}
\label{lemma:equivalence fixed point}
    Suppose Assumptions~\ref{ass:kernel} and~\ref{ass:data} hold.
    \begin{enumerate}
        \item[(i)] If $M \in \rmC([0,T];L^1_\mu)$ solves \eqref{eq:ODE fixed point} and $\alpha^*$ is the best response to $m^M$ given by Proposition~\ref{prop:best-response}, then $\E[Y^{\alpha^*}_t]=M_t$, and hence $\E[X^{\alpha^*}_t] = m^M_t$, for every $t \in [0,T]$.
        \item[(ii)] If $(\hat\alpha,\mathfrak{m})$ is a solution of the mean field game, then $M_t := \E[Y^{\hat\alpha}_t]$ defines an element of $\rmC([0,T];L^1_\mu)$ which solves \eqref{eq:ODE fixed point}, and $\mathfrak{m} = m^M$.
    \end{enumerate}
\end{lemma}
\begin{proof}
    Let $m \in \rmC([0,T];\real^d)$ and let $\alpha^*$ be the best response to $m$. Since $Y^{\alpha^*}$ has continuous sample paths in $L^1_\mu$ and satisfies \eqref{eq:lifted moments}, the map %
    $t\mapsto\E[Y^{\alpha^*}_t]$ belongs to $\rmC([0,T];L^1_\mu)$. Moreover, it satisfies \eqref{eq:fixed_point_frozen}.

    (i) Take $m=m^M$. Subtracting \eqref{eq:ODE fixed point} from that equation, the terms involving $\Lambda^{m}$ cancel, and so does $Bg_0(s)-Bm^M_s+B\mu[M_s]$. Hence $\Delta:=\E[Y^{\alpha^*}]-M$ satisfies
    $$
    \Delta_t(\theta) = \int_0^t e^{-\theta(t-s)}\Big\{B\mu[\Delta_s] - C\hat R(\Gamma_s)^{-1}\int_0^\infty S(\Gamma_s)(\tilde\theta)\Delta_s(\tilde\theta)\mu(\de\tilde\theta)\Big\}\de s .
    $$
    Integrating against $\mu$ and using Tonelli's theorem and Lemmas~\ref{lemma:hat_N_bdd} and~\ref{lemma:S_Gamma_bdd}, we get \[
    \norm{\Delta_t}_{L^1_\mu} \leq (|B|+c_Rc_S|C|)\int_0^t K(t-s)\norm{\Delta_s}_{L^1_\mu}\de s,\]
    and the convolutional Gronwall inequality gives $\Delta\equiv0$. Therefore $\E[X^{\alpha^*}_t]=g_0(t)+\mu[M_t]=m^M_t$.

    (ii) Take $m=\mathfrak{m}$ and $\alpha^*=\hat\alpha$, so that $M=\E[Y^{\hat\alpha}]\in\rmC([0,T];L^1_\mu)$. The consistency condition gives $g_0+\mu[M]=\E[X^{\hat\alpha}]=\mathfrak{m}$, that is $\mathfrak{m}=m^M$. Then $Bg_0(s)-B\mathfrak{m}_s+B\mu[M_s]=0$, and \eqref{eq:fixed_point_frozen} becomes \eqref{eq:ODE fixed point}.
\end{proof}
Solving \eqref{eq:ODE fixed point} is the fixed point step of the mean field game. We establish existence and uniqueness by a contraction argument on the whole interval $[0,T]$, based on the Lipschitz estimate for $\psi$ below. Its proof relies on a detailed analysis of the Riccati equations, the main ingredient being the Lipschitz continuity of the solution map $m \mapsto \Lambda^m$; these technical results are collected in Appendix~\ref{appendix:B}.

\begin{lemma}\label{lemma:psi_lip}
    Suppose Assumptions~\ref{ass:kernel} and~\ref{ass:data} hold. Then, for every $M^1,M^2 \in \rmC([0,T];L^1_\mu)$,
    \begin{equation}
        \sup_{t\in[0,T]}|\psi_t(M^1)-\psi_t(M^2)| \leq \Lip_\psi\, \norm{M^1-M^2}_{\infty}, \qquad \Lip_\psi := c_R|C|\big(\Lip_{h,m}+\Lip_{h,\ell}\,c_K+c_S\big),
    \end{equation}
    with the constants of Lemmas~\ref{lemma:hat_N_bdd}, \ref{lemma:S_Gamma_bdd}, \ref{lemma:h_lip} and~\ref{lemma:Lambda_lip}. In particular, $\Lip_\psi$ does not depend on $g_0$, $\beta$, $\gamma$ and $L$.
\end{lemma}
\begin{proof}
    Set $m^i := m^{M^i}$ and $\Lambda^i := \Lambda^{m^i}$, $i=1,2$. Since $m^1-m^2=\mu[M^1-M^2]$, we have $\norm{m^1-m^2}_\infty \leq \norm{M^1-M^2}_\infty$. By definition of $\psi$ and Lemmas~\ref{lemma:hat_N_bdd} and~\ref{lemma:S_Gamma_bdd},
    $$
    |\psi_t(M^1)-\psi_t(M^2)| \leq c_R|C|\Big\{\big|h(t,\Gamma_t,\Lambda^1_t,m^1)-h(t,\Gamma_t,\Lambda^2_t,m^2)\big| + c_S\norm{M^1_t-M^2_t}_{L^1_\mu}\Big\}.
    $$
    By Lemmas~\ref{lemma:h_lip} and~\ref{lemma:Lambda_lip},
    \begin{align*}
    \big|h(t,\Gamma_t,\Lambda^1_t,m^1)-h(t,\Gamma_t,\Lambda^2_t,m^2)\big| &\leq \Lip_{h,m}\norm{m^1-m^2}_\infty+\Lip_{h,\ell}\norm{\Lambda^1_t-\Lambda^2_t}_{L^1_\mu} \\&
    \leq \big(\Lip_{h,m}+\Lip_{h,\ell}\,c_K\big)\norm{m^1-m^2}_\infty .
    \end{align*}
    Combining the two estimates with $\norm{m^1-m^2}_\infty\leq\norm{M^1-M^2}_\infty$ and $\norm{M^1_t-M^2_t}_{L^1_\mu}\leq\norm{M^1-M^2}_\infty$ gives the claim.
\end{proof}

We now have all the ingredients to establish the well-posedness of the fixed point equation.

\begin{theorem}[Well-posedness of the fixed point equation]
\label{thm:fixed point}
    Suppose Assumptions~\ref{ass:kernel} and~\ref{ass:data} hold, let $\Lip_\psi$ be the constant of Lemma~\ref{lemma:psi_lip}, and set
    \begin{equation}
    \label{eq:smallness}
        \rho_T := \Lip_\psi\,\norm{K}_{L^1(0,T)}.
    \end{equation}
    If $\rho_T<1$, there exists a unique solution $\mathfrak{M} \in \rmC([0,T];L^1_\mu)$ to the fixed point equation \eqref{eq:ODE fixed point}.
\end{theorem}
\begin{proof}
    We apply the Banach fixed point theorem to the map
    $$
    \mathcal{T}(M)_t(\theta) := \int_0^t e^{-\theta(t-s)}\big[\beta(s)+\psi_s(M)\big]\de s, \qquad M \in \rmC([0,T];L^1_\mu),
    $$
    on $\rmC([0,T];L^1_\mu)$ endowed with the norm $\norm{\cdot}_\infty$.

    \medskip
    \emph{Step 1: $\mathcal{T}$ maps $\rmC([0,T];L^1_\mu)$ into itself.} Fix $M$ and set $\Phi_s := \beta(s)+\psi_s(M)$. By Lemma~\ref{lemma:psi_lip}, $|\psi_s(M)|\leq|\psi_s(0)|+\Lip_\psi\norm{M}_\infty$, where $\psi_s(0)=-C\hat R(\Gamma_s)^{-1}\big\{C^\top\mu[\Lambda^{g_0}_s]+F^\top\mu^{\otimes2}[\Gamma_s]\gamma(s)\big\}$ is bounded on $[0,T]$, since $\Lambda^{g_0}\in\rmC([0,T];L^1_\mu)$ by Proposition~\ref{prop:gamma}. Hence $\norm{\Phi}_\infty := \sup_{s\in[0,T]}|\Phi_s|<\infty$. By Tonelli's theorem and the representation $K(r)=\int_0^\infty e^{-\theta r}\mu(\de\theta)$, for every $t \in [0,T]$,
    $$
    \norm{\mathcal{T}(M)_t}_{L^1_\mu} \leq \int_0^t K(t-s)|\Phi_s|\de s \leq \norm{\Phi}_\infty\norm{K}_{L^1(0,T)}.
    $$
    For $0 \leq u \leq t \leq T$, using $e^{-\theta(t-s)}=e^{-\theta(t-u)}e^{-\theta(u-s)}$ for $s \leq u$,
    $$
    \mathcal{T}(M)_t(\theta)-\mathcal{T}(M)_u(\theta) = \int_u^t e^{-\theta(t-s)}\Phi_s\de s - \big(1-e^{-\theta(t-u)}\big)\int_0^u e^{-\theta(u-s)}\Phi_s\de s .
    $$
    Since $\int_0^\infty e^{-\theta(u-s)}\big(1-e^{-\theta(t-u)}\big)\mu(\de\theta)=K(u-s)-K(t-s)\geq 0$ by the monotonicity of $K$, Tonelli's theorem gives
    $$
    \norm{\mathcal{T}(M)_t-\mathcal{T}(M)_u}_{L^1_\mu} \leq \norm{\Phi}_\infty\left(\int_0^{t-u}K(r)\de r + \int_0^u K(r)\de r - \int_{t-u}^t K(r)\de r\right) \leq 2\norm{\Phi}_\infty\int_0^{t-u}K(r)\de r,
    $$
    which vanishes as $t-u \to 0$, since $K \in L^1(0,T)$. Hence $\mathcal{T}(M) \in \rmC([0,T];L^1_\mu)$. 

    \medskip
    \emph{Step 2: contraction.} Fix $M^1,M^2 \in \rmC([0,T];L^1_\mu)$. By Lemma~\ref{lemma:psi_lip},
    $$
    |\mathcal{T}(M^1)_t(\theta)-\mathcal{T}(M^2)_t(\theta)| \leq \int_0^t e^{-\theta(t-s)}|\psi_s(M^1)-\psi_s(M^2)|\de s \leq \Lip_\psi\norm{M^1-M^2}_\infty\int_0^t e^{-\theta(t-s)}\de s .
    $$
    Integrating against $\mu$ and applying Tonelli's theorem,
    $$
    \norm{\mathcal{T}(M^1)-\mathcal{T}(M^2)}_\infty \leq \Lip_\psi\norm{K}_{L^1(0,T)}\norm{M^1-M^2}_\infty = \rho_T\norm{M^1-M^2}_\infty .
    $$
    Since $\rho_T<1$, $\mathcal{T}$ is a contraction, and the Banach fixed point theorem yields the claim.
\end{proof}
\begin{remark}
\label{rmk:rho}
        The constant $\rho_T$ depends only on $\mu$, on the matrices $B,C,D,F,Q,R$ and on the horizon $T$, through the kernel $K$ and the solution $\Gamma$ of the first Riccati equation. In particular, it does not depend on $g_0$, $\beta$, $\gamma$ and $L$. Moreover, $T\mapsto\rho_T$ is non-decreasing and $\rho_T\to0$ as $T\downarrow0$, see Remark~\ref{rmk:horizon dependence}. Hence, setting $T_0:=\sup\{T>0:\rho_T<1\}\in(0,+\infty]$, the condition $\rho_T<1$ holds for every $T<T_0$, and Theorems~\ref{thm:fixed point} and~\ref{thm:mfg-wellposed} apply on every such horizon, whatever the choice of $g_0$, $\beta$, $\gamma$ and $L$.
\end{remark}
When $\rho_T<1$, we denote by $\mathfrak{m}_t = g_0(t) + \mu[\mathfrak{M}_t]$ the projection of the unique solution of the fixed point equation \eqref{eq:ODE fixed point}.

\medskip

Let $\hat\alpha$ be the best response to $\mathfrak{m}$ given by Proposition~\ref{prop:best-response}. By Lemma~\ref{lemma:equivalence fixed point}(i), the pair $(\hat\alpha,\mathfrak{m})$ solves the mean field game in the sense of Definition~\ref{def:MFG-solution}. It is the unique solution: if $(\tilde\alpha,\tilde{\mathfrak{m}})$ is another one, Lemma~\ref{lemma:equivalence fixed point}(ii) provides a solution $\tilde M$ of \eqref{eq:ODE fixed point} with $\tilde{\mathfrak{m}}=g_0+\mu[\tilde M]$; by Theorem~\ref{thm:fixed point}, $\tilde M=\mathfrak{M}$, hence $\tilde{\mathfrak{m}}=\mathfrak{m}$, and $\tilde\alpha=\hat\alpha$ by the uniqueness in Proposition~\ref{prop:best-response}. We have thus proved the first main result of the paper: 
\mfg*

\begin{remark}[Affine structure of the equilibrium controls]
\label{rmk:feedback}
    Let $(\hat\alpha,\mathfrak{m})$ be any solution of the mean field game. By Definition~\ref{def:MFG-solution}, $\hat\alpha$ is the best response to the continuous curve $\mathfrak{m}$, so that, by the uniqueness in Proposition~\ref{prop:best-response}, it is given, up to $\de t\otimes\de\Prob$-null sets, by the feedback \eqref{eq:LQ-MFG feedback control} with $m=\mathfrak{m}$:
    \[
    \hat\alpha_t = -\hat{R}(\Gamma_t)^{-1}\left\{ h(t,\Gamma_t,\Lambda^{\mathfrak{m}}_t,\mathfrak{m}) + \int_0^\infty S(\Gamma_t)(\theta)\,Y^{\hat\alpha}_t(\theta)\,\mu(\de\theta) \right\}.
    \]
    This is an affine feedback on the lifted state: the linear gain depends only on $\Gamma$, which does not depend on $\mathfrak{m}$, while the offset $h(t,\Gamma_t,\Lambda^{\mathfrak{m}}_t,\mathfrak{m})$ carries the dependence on the equilibrium mean. This representation holds for every solution, regardless of uniqueness and of the condition $\rho_T<1$. In general, the feedback is Markovian in the lifted state $Y^{\hat\alpha}_t$, but not in $X^{\hat\alpha}_t$, since $Y^{\hat\alpha}_t$ depends on the whole past of the dynamics.
\end{remark}

\subsection{Arbitrary horizons without control in the diffusion}
\label{subsec:arbitrary horizons}
The smallness condition $\rho_T<1$ of Theorem~\ref{thm:fixed point} restricts the horizon. In this subsection we remove it when the control does not act on the diffusion coefficient, that is $F=0$, and the interaction pulls each player towards the population mean.

\begin{theorem}[Arbitrary horizons, no control in the diffusion]
\label{thm:global F=0}
    Suppose Assumptions~\ref{ass:kernel} and~\ref{ass:data} hold, $F=0$, and $B=-P$ for some $P\in S^d_+$ commuting with $CR^{-1}C^\top$. Then, for every $T>0$, the mean field game admits a unique solution.
\end{theorem}

In dimension one the commutation condition is automatic, and we obtain the following.

\begin{corollary}
\label{cor:global d1}
    Suppose Assumptions~\ref{ass:kernel} and~\ref{ass:data} hold, $d=1$, $F=0$ and $B\leq0$. Then, for every $T>0$, the mean field game admits a unique solution.
\end{corollary}

The proof exploits the linear structure of the problem and consists of three steps: (i) since the data enter the fixed point equation \eqref{eq:ODE fixed point} linearly, the Fredholm alternative reduces existence and uniqueness for arbitrary data to uniqueness for vanishing data; (ii) for vanishing data, the equilibrium mean solves, together with the mean of the adjoint variable of the representative player, a linear forward-backward system of Volterra equations; (iii) pairing the two equations of this system gives an energy identity, and the sign condition on $B$, combined with the positivity of completely monotone kernels, forces the equilibrium mean to vanish.

\medskip
We start with the reduction to vanishing data. We call \emph{homogeneous} the mean field game with $g_0\equiv\beta\equiv\gamma\equiv0$ and $L=0$, the other coefficients being unchanged. In this case, the mappings $h$ and $\mathcal{R}_2$ take the following form:
\begin{align*}
    h(t,G,\ell,m) & = C^\top \mu[\ell] + F^\top \mu^{\otimes 2}[G] (-Dm_t) \\
    \mathcal{R}_2(t,G,\ell,m)(\theta) & =  B^\top \mu[\ell] +  \mu[G(\theta,\cdot)](-Bm_t) + D^\top \mu^{\otimes 2}[G](-Dm_t)  - S(G)(\theta)^\top \hat{R}(G)^{-1}h(t,G,\ell,m).
\end{align*} 
The homogeneous game always admits the trivial solution $(\hat\alpha,\mathfrak m)=(0,0)$: for $m\equiv0$, the control $\alpha\equiv0$ produces $X^\alpha\equiv0$ and zero cost, while the cost is non-negative, so that $\alpha\equiv0$ is the best response to $m\equiv0$ and the consistency condition holds.
Note also that the solution $\Gamma$ of the first Riccati equation in \eqref{eq:Riccati} depends neither on the curve $m$ nor on $g_0,\beta,\gamma,L$, since $\mathcal R_1$ only involves the matrices of the problem. Hence $\Gamma$, and with it $\hat R(\Gamma_t)$, $S(\Gamma_t)$ and the constant $\Lip_\psi$ of Lemma~\ref{lemma:psi_lip}, is the same for the homogeneous game and for the original one. The following result holds for any $F$.

\begin{proposition}
\label{prop:fredholm}
    Suppose Assumptions~\ref{ass:kernel} and~\ref{ass:data} hold, and that the homogeneous mean field game admits no solution other than the trivial one. Then, for every choice of $g_0,\beta,\gamma,L$ satisfying Assumption~\ref{ass:data}, the mean field game admits a unique solution.
\end{proposition}

\begin{proof} The proof is divided into three steps.\\
    \emph{Step 1: reformulation as a compact perturbation of the identity.} Let us define $\mathcal V\colon B_b([0,T];\real^d)\to\rmC([0,T];L^1_\mu)$ by $(\mathcal V\varphi)_t(\theta):=\int_0^te^{-\theta(t-s)}\varphi_s\de s$, which is linear, and compact by Lemma~\ref{lemma:V compact}. Thus, \eqref{eq:ODE fixed point} reads $M=\mathcal V(\beta+\psi(M))$. For fixed $G$, the maps $h(t,G,\ell,m)$ and $\mathcal R_2(t,G,\ell,m)$ are linear in $(\ell,m,g_0,\beta,\gamma,L)$ jointly. Since the second Riccati equation in \eqref{eq:Riccati} is linear and its solution is unique by Proposition~\ref{prop:gamma}, $\Lambda^m$ depends linearly on $(m,g_0,\beta,\gamma,L)$. 
    As $m^M=g_0+\mu[M]$, and since $\hat R(\Gamma_t)$ and $S(\Gamma_t)$ only depend on $\Gamma$, we get $\psi(M)=\psi(0)+\psi^0(M)$, where
    \begin{align*}
        \psi_t(0) & = -C\hat{R}(\Gamma_t)^{-1}\big\{ C^\top\mu[\Lambda^{g_0}_t]+F^\top\mu^{\otimes 2}[\Gamma_t]\gamma(t) \big\}, \\
        \psi^0_t(M) &= -C\hat{R}(\Gamma_t)^{-1}\Big\{ C^\top\mu[\Lambda^{0,\mu[M]}_t]-F^\top\mu^{\otimes 2}[\Gamma_t]D\mu[M_t] + \int_0^\infty S(\Gamma_t)(\theta)M_t(\theta)\mu(\de\theta) \Big\},
    \end{align*}
    and $\Lambda^{0,m}$ denotes the solution of the second Riccati equation of the homogeneous game associated with the curve $m$, so that $\Lambda^{m^M}=\Lambda^{g_0}+\Lambda^{0,\mu[M]}$ by linearity. In particular, $\psi^0$ is the operator \eqref{eq:definition of psi} of the homogeneous game, and it is linear. Hence \eqref{eq:ODE fixed point} reads
    $$
    M=c+\mathcal AM,\qquad c:=\mathcal V\big(\beta+\psi(0)\big),\qquad \mathcal A:=\mathcal V\circ\psi^0 .
    $$
    By Lemma~\ref{lemma:psi_lip} applied to the homogeneous game, $\psi^0$ is a bounded linear map from $\rmC([0,T];L^1_\mu)$ to the space of bounded measurable functions with the sup norm, so that $\mathcal A=\mathcal V\circ\psi^0$ is compact on $\rmC([0,T];L^1_\mu)$. By the Fredholm alternative \cite[Theorem~6.6]{Brezis2011}, the equation $M=c+\mathcal AM$ has exactly one solution for every $c\in\rmC([0,T];L^1_\mu)$, provided that the homogeneous equation $M=\mathcal AM$ only has the trivial solution.

    \smallskip

    \emph{Step 2: $M=\mathcal AM$ implies $M=0$.} Let $M=\mathcal AM$, that is, $M$ solves \eqref{eq:ODE fixed point} for the homogeneous game, and let $\alpha^*$ be the best response to $\mu[M]$. By Lemma~\ref{lemma:equivalence fixed point}(i), $(\alpha^*,\mu[M])$ is a solution of the homogeneous game and $\E[Y^{\alpha^*}]=M$. By the assumption of the proposition, this solution is the trivial one, so that $\alpha^*\equiv0$. Since the data also vanish, $Y^{\alpha^*}\equiv0$, and therefore $M=\E[Y^{\alpha^*}]=0$.
    
    \smallskip
    
    \emph{Step 3: conclusion.} By Steps~1 and~2, \eqref{eq:ODE fixed point} has a unique solution $\mathfrak M$ for every choice of the data. The best response to $\mathfrak m:=g_0+\mu[\mathfrak M]$ provides a solution of the mean field game by Lemma~\ref{lemma:equivalence fixed point}(i), and any other solution coincides with it by Lemma~\ref{lemma:equivalence fixed point}(ii), the uniqueness of $\mathfrak M$ and Proposition~\ref{prop:best-response}.
\end{proof}

It remains to verify the assumption of Proposition~\ref{prop:fredholm}, namely that the homogeneous game only has the trivial solution. To this end, we fix an arbitrary solution of the homogeneous game, derive a closed system of equations for its mean, and then show, by an energy argument, that this mean vanishes.

From now on $F=0$, so that $\hat R(G)=R$, $S(G)(\tau)=C^\top\mu[G(\cdot,\tau)]$ and $h(t,G,\ell,m)=C^\top\mu[\ell]$. We set $\mathcal G:=CR^{-1}C^\top\in S^d_+$ and, for a bounded measurable $f\colon[0,T]\to\real^d$, we define the forward and backward convolutions
$$
(K*f)_t:=\int_0^tK(t-s)f_s\de s,\qquad (K^\star f)_t:=\int_t^TK(s-t)f_s\de s .
$$
By \eqref{eq:LQ-MFG feedback control} and Fubini's theorem, for a given curve $m$ the best response reads
$$
\alpha^*_t=-R^{-1}C^\top\int_0^\infty\Big(\Lambda^m_t(\theta)+\int_0^\infty\Gamma_t(\theta,\tau)Y_t(\tau)\mu(\de\tau)\Big)\mu(\de\theta),
$$
where the function of $\theta$ in brackets plays the role of the adjoint variable of the lifted control problem. The function $y$ in the next lemma is its mean, averaged against $\mu$. The lemma shows that the pair formed by the equilibrium mean and $y$ solves a linear forward-backward system in which $\Gamma$ no longer appears.

\begin{lemma}
\label{lemma:homogeneous system}
    Suppose Assumptions~\ref{ass:kernel} and~\ref{ass:data} hold and $F=0$. Let $(\hat\alpha,\mathfrak m)$ be a solution of the homogeneous mean field game, set $M_t:=\E[Y^{\hat\alpha}_t]$ and
    $$
    \widetilde\Lambda_t(\theta):=\Lambda^{\mathfrak m}_t(\theta)+\int_0^\infty\Gamma_t(\theta,\tau)M_t(\tau)\mu(\de\tau),\qquad y_t:=\mu[\widetilde\Lambda_t].
    $$
    Then $x:=\mathfrak m$ and $y$ are bounded measurable functions such that
    \begin{equation}
    \label{eq:homogeneous system}
        x=-K*(\mathcal Gy),\qquad y=K^\star\big(Qx+B^\top y\big).
    \end{equation}
\end{lemma}

\begin{proof}
    By Lemma~\ref{lemma:equivalence fixed point}(ii), $M\in\rmC([0,T];L^1_\mu)$, $\mathfrak m=\mu[M]$, and $M$ solves \eqref{eq:ODE fixed point} with $\beta\equiv0$. We write $g_t(\theta):=\mu[\Gamma_t(\theta,\cdot)]$. By the symmetry of $\Gamma_t$, $\mu[\Gamma_t(\cdot,\tau)]=g_t(\tau)^\top$, hence $S(\Gamma_t)(\theta)^\top=g_t(\theta)C$.

    \smallskip

    We first derive the first equation in \eqref{eq:homogeneous system}. By Fubini's theorem and the definition of $\widetilde\Lambda_t$, we have $$\int_0^\infty S(\Gamma_t)(\tau)M_t(\tau)\mu(\de\tau)=C^\top\mu[\widetilde\Lambda_t-\Lambda^{\mathfrak m}_t].$$ Since $h(t,\Gamma_t,\Lambda^{\mathfrak m}_t,\mathfrak m)=C^\top\mu[\Lambda^{\mathfrak m}_t]$, this gives $\psi_t(M)=-\mathcal Gy_t$, and \eqref{eq:ODE fixed point} gives $M_t(\theta)=-\int_0^te^{-\theta(t-s)}\mathcal Gy_s\de s$. Integrating against $\mu$ gives $x=-K*(\mathcal Gy)$. Note that $y$ is bounded, since $\Lambda^{\mathfrak m}\in\rmC([0,T];L^1_\mu)$ and $|\mu[\widetilde\Lambda_t-\Lambda^{\mathfrak m}_t]|\leq c_\Gamma\norm{M_t}_{L^1_\mu}$ by the symmetry of $\Gamma_t$ and Proposition~\ref{prop:gamma}.

    \smallskip

    We now derive the second equation, by differentiating the product $\Gamma_s(\theta,\tau)M_s(\tau)$. The map $s\mapsto\mathcal R_1(\Gamma_s)(\theta,\tau)$ is bounded, uniformly in $s$ and $\mu\otimes\mu$-a.e., by Proposition~\ref{prop:gamma} and Lemmas~\ref{lemma:hat_N_bdd} and~\ref{lemma:S_Gamma_bdd}. Hence, by Definition~\ref{def:sol_riccati}, for $\mu\otimes\mu$-a.e.\ $(\theta,\tau)$ the map $s\mapsto\Gamma_s(\theta,\tau)$ is absolutely continuous with derivative $(\theta+\tau)\Gamma_s(\theta,\tau)-\mathcal R_1(\Gamma_s)(\theta,\tau)$, while $s\mapsto M_s(\tau)$ is absolutely continuous with derivative $-\tau M_s(\tau)-\mathcal Gy_s$. For fixed $t$, the terms involving $\theta$ and $\tau$ cancel in
    $$
    \frac{\de}{\de s}\Big(e^{-\theta(s-t)}\Gamma_s(\theta,\tau)M_s(\tau)\Big)=-e^{-\theta(s-t)}\Big(\mathcal R_1(\Gamma_s)(\theta,\tau)M_s(\tau)+\Gamma_s(\theta,\tau)\mathcal Gy_s\Big).
    $$
    We integrate over $[t,T]$, using $\Gamma_T=0$, then in $\tau$ against $\mu$, which is allowed by Fubini's theorem and the bounds above, and we add the mild equation for $\Lambda^{\mathfrak m}$. This gives $\widetilde\Lambda_t(\theta)=\int_t^Te^{-\theta(s-t)}\Phi_s(\theta)\de s$, with
    $$
    \Phi_s(\theta):=\mathcal R_2(s,\Gamma_s,\Lambda^{\mathfrak m}_s,\mathfrak m)(\theta)+\int_0^\infty\mathcal R_1(\Gamma_s)(\theta,\tau)M_s(\tau)\mu(\de\tau)+g_s(\theta)\mathcal Gy_s .
    $$
    Since the data vanish and $\mathfrak m=x$, the definitions of $\mathcal R_1$ and $\mathcal R_2$ give
    \begin{align*}
        \int_0^\infty\mathcal R_1(\Gamma_s)(\theta,\tau)M_s(\tau)\mu(\de\tau)&=\big(Q+D^\top\mu^{\otimes2}[\Gamma_s]D\big)x_s+B^\top\big(y_s-\mu[\Lambda^{\mathfrak m}_s]\big)+g_s(\theta)Bx_s-g_s(\theta)\mathcal G\big(y_s-\mu[\Lambda^{\mathfrak m}_s]\big),\\
        \mathcal R_2(s,\Gamma_s,\Lambda^{\mathfrak m}_s,\mathfrak m)(\theta)&=B^\top\mu[\Lambda^{\mathfrak m}_s]-g_s(\theta)Bx_s-D^\top\mu^{\otimes2}[\Gamma_s]Dx_s-g_s(\theta)\mathcal G\mu[\Lambda^{\mathfrak m}_s].
    \end{align*}
    Adding up, all the terms involving $\Gamma_s$ and $\Lambda^{\mathfrak m}_s$ cancel, and $\Phi_s(\theta)=Qx_s+B^\top y_s$. Integrating $\widetilde\Lambda_t$ against $\mu$ gives the second equation in \eqref{eq:homogeneous system}.
\end{proof}

We now have all the ingredients to prove Theorem~\ref{thm:global F=0}, by means of an energy identity for \eqref{eq:homogeneous system}. To explain the idea, consider first the Markovian case $K\equiv1$. Then \eqref{eq:homogeneous system} is the system of ordinary differential equations
$$
\tfrac{\de x_t}{\de t}=-\mathcal Gy_t,\quad x_0=0,\qquad \tfrac{\de y_t}{\de t}=-Qx_t-B^\top y_t,\quad y_T=0 .
$$
Integrating $\frac{\de}{\de t}(x_t^\top y_t)$ over $[0,T]$ and using the boundary conditions and $B^\top=-P$, we get
$$
\int_0^T\big(x_t^\top Qx_t+y_t^\top\mathcal Gy_t\big)\de t=\int_0^Tx_t^\top Py_t\,\de t=-\frac12\Big|\int_0^TNy_t\,\de t\Big|^2 .
$$
Here $N:=(P\mathcal G)^{1/2}$, which is well defined since $P$ and $\mathcal G$ are commuting symmetric positive semi-definite matrices, and the last equality follows by inserting $x_t=-\int_0^t\mathcal Gy_s\de s$. The left-hand side is non-negative and the right-hand side is non-positive, so that both vanish; hence $\mathcal Gy\equiv0$ and $x\equiv0$, which in turn implies that the solution is trivial. For a general kernel the system is no longer differential, and the same identity is obtained in integrated form, using that $K^\star$ is the adjoint of $K*$ in $L^2(0,T;\real^d)$. The sign of the right-hand side then follows from the positivity of completely monotone kernels.

\begin{proof}[Proof of Theorem~\ref{thm:global F=0}]
    By Proposition~\ref{prop:fredholm}, it suffices to show that the homogeneous game has no solution other than the trivial one. Let $(\hat\alpha,\mathfrak m)$ be a solution of the homogeneous game, and let $x,y$ be as in Lemma~\ref{lemma:homogeneous system}. By Fubini's theorem, $K^\star$ is the adjoint of $K*$ in $L^2(0,T;\real^d)$, whose inner product we denote by $\langle\cdot,\cdot\rangle$. Pairing the first equation in \eqref{eq:homogeneous system} with $Qx$ and using the second one, we get
    $$
    \langle Qx,x\rangle=-\langle Qx,K*(\mathcal Gy)\rangle=-\langle K^\star(Qx),\mathcal Gy\rangle=-\langle y-K^\star(B^\top y),\mathcal Gy\rangle,
    $$
    that is, since $B^\top=-P$,
    $$
    \langle Qx,x\rangle+\langle\mathcal Gy,y\rangle=-\langle K^\star(Py),\mathcal Gy\rangle .
    $$
    The left-hand side is non-negative, while, with $N$ as above, the right-hand side equals $-\langle K^\star(Ny),Ny\rangle=-\langle Ny,K*(Ny)\rangle\leq0$ by Proposition~\ref{prop:positive definite}. Hence both sides vanish. In particular $\langle\mathcal Gy,y\rangle=0$, so that $\mathcal Gy=0$ since $\mathcal G$ is positive semi-definite, and the first equation in \eqref{eq:homogeneous system} gives $\mathfrak m=x=-K*(\mathcal Gy)=0$. Finally, $\hat\alpha$ is the best response to $\mathfrak m\equiv0$ in the homogeneous game, which is $\alpha\equiv0$ as observed before Proposition~\ref{prop:fredholm}; by the uniqueness in Proposition~\ref{prop:best-response}, $\hat\alpha\equiv0$. Hence $(\hat\alpha,\mathfrak m)=(0,0)$.
\end{proof}

\begin{remark}
\label{rmk:global F=0}
    Theorem~\ref{thm:global F=0} holds for every kernel satisfying Assumption~\ref{ass:kernel}, singular ones included, and it involves no smallness condition. Some condition on $B$ is needed in general: for $d\geq2$ one can construct explicit examples, already with an exponential kernel or with $K\equiv1$, in which, on a specific horizon, uniqueness fails and, for suitable data, existence fails as well. This reflects a known feature of linear-quadratic mean field games with an interaction in the drift, see \cite[Section~3.5]{CarmonaDelarue} and \cite{BensoussanSungYamYung2016}.
\end{remark}

\begin{remark}
\label{rmk:global F}
    The same strategy applies when $F\neq0$, at the price of additional technicalities, and we only sketch it. Proposition~\ref{prop:fredholm} holds for any $F$, and the computation of Lemma~\ref{lemma:homogeneous system} leads to the system
    $$
    x=-K*(\mathcal G y),\qquad y=K^\star\big(Qx+\widetilde B^\top y\big),
    $$
    where now $\mathcal G_t:=C\hat R(\Gamma_t)^{-1}C^\top$ and $\widetilde B_t^\top:=B^\top-D^\top\mu^{\otimes2}[\Gamma_t]F\hat R(\Gamma_t)^{-1}C^\top$ depend on time. The energy identity becomes $\langle Qx,x\rangle+\langle\mathcal Gy,y\rangle=\langle K^\star(\widetilde B^\top y),\mathcal Gy\rangle$, whose right-hand side no longer has a sign, even when $B=-\varrho I_d$. It can however be treated as a perturbation of the case $F=0$: using that $\hat R(\Gamma_t)-R\in S^q_+$ with $|\hat R(\Gamma_t)-R|\leq|F|^2\,|\mu^{\otimes2}[\Gamma_t]|$, one obtains uniqueness for the homogeneous game, hence well-posedness, as soon as $B=-\varrho I_d$ for some $\varrho\geq0$ and
    $$
    \frac{\norm{K}_{L^1(0,T)}\,\big(|R|+|F|^2\bar\gamma_T\big)\,\bar\gamma_T\,|F|\,\big(\varrho|F|+|C||D|\big)}{\lambda^2}<1,\qquad \bar\gamma_T:=\sup_{t\in[0,T]}\big|\mu^{\otimes2}[\Gamma_t]\big| .
    $$
    This condition is always satisfied when $F=0$, in agreement with Theorem~\ref{thm:global F=0}, and more generally when $F$ is small enough compared with the other coefficients.
\end{remark}

\section{Propagation of chaos and approximated Nash equilibrium}
\label{sec:propagation}

Throughout this section, we fix a solution $(\hat\alpha,\mathfrak{m})$ of the mean field game, whose control $\hat\alpha$ is in affine feedback form by Remark~\ref{rmk:feedback}, and we work in dimension $d=q=1$ to lighten the notation; the general case is identical up to keeping track of the matrix dimensions. We show that this feedback, applied by each player to its own lifted state, provides approximate equilibria for the $N$-player game. The argument rests on a propagation of chaos property: as $N \to \infty$, each player becomes asymptotically independent of the others and close to its mean field counterpart. Section~\ref{ssec:convergence} makes this precise and provides the rate of convergence, while Section~\ref{ssec:nash} uses it to establish Theorem~\ref{thm:eps-nash}.

\subsection{Propagation of chaos}\label{ssec:convergence}
In this subsection we establish the propagation of chaos for the $N$-player game: as the population grows, each player becomes close to an independent copy of the mean field dynamics. We first recast the coupled particle system as a single Volterra equation, which we lift to the Markovian space $L^1_\mu$, and we introduce the mean field replicas against which the players are compared.

Fix a sequence of one-dimensional independent Brownian motions $(W^i)_{i \geq 1}$ and for every $N \in \Nat$ consider the particle system described by the following SVEs, for every $1 \leq i \leq N$
\begin{equation}
\label{eq:particles SVE}
    \begin{aligned}
        X_t^{i,N} & = g_0(t) +  \int_0^t K(t-s)\left[\beta(s) + B   \left(X_s^{i,N} - \frac{1}{N} \sum_{j=1}^N  X^{j,N}_s  \right) + C\alpha_s^{i,N}\right]\de s\\[0.2cm]
        & \qquad + \int_0^t K(t-s)\left[\gamma(s) + D\left(X_s^{i,N} - \frac{1}{N} \sum_{j=1}^N  X^{j,N}_s \right) + F\alpha_s^{i,N}\right]\de W^i_s, \quad t \in [0,T].
    \end{aligned}
\end{equation}
Although each player is driven by its own noise, the empirical mean couples all the equations. It is therefore convenient to treat the whole population as a single Volterra process: we rewrite the former system as an SVE for the $N$-dimensional process $X:=(X^{i,N})_{1\leq i\leq N}$. Let us define the following $N \times N$ matrix
$$
 \pi = \mathbb{I} - \frac{1}{N} \mathbf{1}\mathbf{1}^\top,
$$
where $\mathbb{I}$ is the $N \times N$ identity matrix and $\mathbf{1}:=(1,\dots,1)^\top\in \real^N$. Setting $\alpha:=(\alpha^{i,N})_{1\leq i\leq N}$, we get
\begin{equation}
\label{eq:particles SVE vector}
    \begin{aligned}
        X_t & = g_0(t) \mathbf{1} + \int_0^t K(t-s)[ \beta(s)\mathbf{1} + B\pi X_s + C \alpha_s]\de s\\
        & \qquad + \int_0^t K(t-s)[ \gamma(s)\mathbb{I} + D\mathrm{diag}(\pi X_s) + F\mathrm{diag}( \alpha_s)]\de W_s
    \end{aligned}
\end{equation}
where $\mathrm{diag}$ is the diagonal operator and $W:=(W^{i})_{1\leq i\leq N}$. As in the previous sections, we can lift \eqref{eq:particles SVE vector} in the space $L^1_\mu$ thanks to the relation $X_t = g_0(t)\mathbf{1} + \mu[Y_t]$, where $Y$ satisfies, for every $t \in [0,T], \mu\text{-a.e.}$
\begin{equation}
\label{eq:particles vector lifted}
    \begin{aligned}
        Y_t(\theta) & = \int_0^t e^{-\theta(t-s)}[ \beta(s)\mathbf{1} + B\pi \mu[Y_s] + C \alpha_s]\de s\\
        & \qquad + \int_0^t e^{-\theta(t-s)}[ \gamma(s)\mathbb{I} + D\mathrm{diag}(\pi \mu[Y_s]) + F\mathrm{diag}( \alpha_s)]\de W_s. 
    \end{aligned}
\end{equation}
For a given admissible profile, the former equation is well-posed by \cite[Theorem~4.4]{LQ_Volterra_AJMP}; with the feedback controls $\alpha^{i,N}$ defined below, which are affine in the lifted states with bounded coefficients, it has Lipschitz coefficients and is well-posed by \cite[Theorem 4.1]{LQ_Volterra_AJMP}. So is \eqref{eq:particles SVE vector}. Moreover, the $i$-th component of $Y$ is the lifted process of $X^{i,N}$ and this justifies our following notation
\begin{equation}
\label{eq:particles lifted}
    \begin{aligned}
        Y^{i,N}_t(\theta) & = \int_0^t e^{-\theta(t-s)}\left[ \beta(s) + B \left( \mu[Y^{i,N}_s] - \frac{1}{N}\sum_{j=1}^N \mu[Y^{j,N}_s] \right) + C \alpha^{i,N}_s\right]\de s\\
        & \qquad + \int_0^t e^{-\theta(t-s)}\left[ \gamma(s) + D \left( \mu[Y^{i,N}_s] - \frac{1}{N}\sum_{j=1}^N \mu[Y^{j,N}_s] \right) + F \alpha^{i,N}_s\right]\de W_s^i. 
    \end{aligned}
\end{equation}
In all the previous equations, $\alpha^{i,N}$ is the equilibrium feedback of Remark~\ref{rmk:feedback} applied to the lifted state of player $i$, that is
$$
\alpha^{i,N}_t := -\hat{R}(\Gamma_t)^{-1} \left\{ h(t,\Gamma_t,\Lambda_t^\mathfrak{m},\mathfrak{m}) + \int_0^\infty S(\Gamma_t)(\theta)Y^{i,N}_t(\theta)\mu(\de\theta) \right\}.
$$
The propagation of chaos compares each player with an independent copy of the mean field dynamics, driven by the same noise. We therefore introduce the $N$ replicas
\begin{equation}
\label{eq: mean field replicas}
    \begin{aligned}
        Y^{i}_t(\theta) & = \int_0^t e^{-\theta(t-s)}\left[ \beta(s) + B \left( g_0(s)+\mu[Y^{i}_s] - \mathfrak{m}_s\right) + C \alpha^{i}_s\right]\de s\\
        & \qquad + \int_0^t e^{-\theta(t-s)}\left[ \gamma(s) + D \left( g_0(s)+\mu[Y^{i}_s] - \mathfrak{m}_s \right) + F \alpha^{i}_s\right]\de W_s^i,
    \end{aligned}
\end{equation}
where
\[
\alpha^{i}_t := -\hat{R}(\Gamma_t)^{-1} \left\{ h(t,\Gamma_t,\Lambda_t^\mathfrak{m},\mathfrak{m}) + \int_0^\infty S(\Gamma_t)(\theta)Y^{i}_t(\theta)\mu(\de\theta) \right\}.
\]
Each $Y^i$ is the optimally controlled lifted state of the mean field problem with frozen curve $\mathfrak{m}$, driven by $W^i$. Hence the processes $Y^i$ are i.i.d.\ and, by the consistency condition, $\E[\mu[Y^i_s]] = \mathfrak{m}_s - g_0(s)$, so that $g_0(s)+\mu[Y^{i}_s]-\mathfrak{m}_s = \mu[Y^{i}_s]-\E[\mu[Y^{i}_s]]$.
The convergence analysis relies on two auxiliary estimates, collected in Appendix~\ref{appendix:convergence tools}: a law of large numbers for empirical means of i.i.d.\ processes (Lemma~\ref{lemma:CLT}) and a moment estimate for lifted Volterra dynamics (Proposition~\ref{prop:usual machinery}).

Before turning to the propagation of chaos, we establish that the moments of the particle systems are bounded uniformly in the number of players. This a priori control is used repeatedly in the sequel. 

\begin{proposition}
\label{prop:a priori estimates}
    Suppose Assumptions~\ref{ass:kernel} and~\ref{ass:data} hold. Then, for every $p=2,3,4$, there exists a constant $c=c(p)>0$, independent of $N$, such that the following hold: for the equilibrium particles and their mean field replicas,
        $$
        \sup_{1\leq i\leq N}\sup_{t\in[0,T]}\E\left[\norm{Y^{i,N}_t}_{L^1_\mu}^p\right] \leq c
        \qquad\text{and}\qquad
        \sup_{t\in[0,T]}\E\left[\norm{Y^{i}_t}_{L^1_\mu}^p\right] \leq c.
        $$
    Moreover, the same estimates hold for the states $X^{i,N}$, $X^{i}$ in place of $Y^{i,N}$, $Y^{i}$ respectively, that is
    $$
    \sup_{1\leq i\leq N}\sup_{t\in[0,T]}\E\left[|X^{i,N}_t|^p\right] \leq c, \qquad
    \sup_{t\in[0,T]}\E\left[|X^{i}_t|^p\right] \leq c.
    $$

\end{proposition}
\begin{proof}
    Fix $p\in\{2,3,4\}$. We prove the bound for $Y^{i,N}$ in detail; the other processes are treated by the same argument, and we only point out what changes.

    We apply Proposition~\ref{prop:usual machinery} to $Y^{i,N}$ with
    $$
    \Delta^1_s = \beta(s) + B\Big(\mu[Y^{i,N}_s] - \tfrac{1}{N}\textstyle\sum_{j=1}^N \mu[Y^{j,N}_s]\Big) + C\alpha^{i,N}_s, \quad
    \Delta^2_s = \gamma(s) + D\Big(\mu[Y^{i,N}_s] - \tfrac{1}{N}\textstyle\sum_{j=1}^N \mu[Y^{j,N}_s]\Big) + F\alpha^{i,N}_s.
    $$
    The optimal control is affine in the lifted state: by Lemma~\ref{lemma:hat_N_bdd} and Lemma~\ref{lemma:S_Gamma_bdd},
    $$
    |\alpha^{i,N}_s| \leq c_R\Big(|h(s,\Gamma_s,\Lambda^{\mathfrak{m}}_s,\mathfrak{m})| + c_S\norm{Y^{i,N}_s}_{L^1_\mu}\Big) \leq c\big(1 + \norm{Y^{i,N}_s}_{L^1_\mu}\big),
    $$
    where we used that $\sup_{s}|h(s,\Gamma_s,\Lambda^{\mathfrak{m}}_s,\mathfrak{m})|<\infty$, 
    since $(\Gamma,\Lambda^{\mathfrak{m}},\mathfrak{m})$ are the continuous functions associated with the fixed solution $(\hat\alpha,\mathfrak{m})$ of the mean field game (Proposition~\ref{prop:gamma}). Using $|\mu[\varphi]|\leq\norm{\varphi}_{L^1_\mu}$, the boundedness of $\beta,\gamma$, and Jensen's inequality for the empirical average, we get
    $$
    \E\left[|\Delta^1_s|^p + |\Delta^2_s|^p\right] \leq c\Big(1 + \E\big[\norm{Y^{i,N}_s}_{L^1_\mu}^p\big] + \tfrac{1}{N}\textstyle\sum_{j=1}^N \E\big[\norm{Y^{j,N}_s}_{L^1_\mu}^p\big]\Big).
    $$
    Set $\phi_N(s):=\sup_{1\leq i\leq N}\E[\norm{Y^{i,N}_s}_{L^1_\mu}^p]$, which is finite for each fixed $N$ by the well-posedness of the particle system \cite[Theorem 4.1]{LQ_Volterra_AJMP}. The right-hand side above is bounded by $c(1+\phi_N(s))$ uniformly in $i$, so Proposition~\ref{prop:usual machinery} yields, for every $i$,
    $$
    \E\left[\norm{Y^{i,N}_t}_{L^1_\mu}^p\right] \leq c\int_0^t \widehat{K}(t-s)\big(1+\phi_N(s)\big)\de s,
    $$
    and taking the supremum over $i$,
    $$
    \phi_N(t) \leq c\norm{\widehat{K}}_{L^1(0,T)} + c\int_0^t \widehat{K}(t-s)\phi_N(s)\de s.
    $$
    Since $\widehat{K}$ and the resolvent $\widehat{r}$ of $-c\widehat{K}$ depend only on $c$, $\mu$, $p$ and $T$, the convolutional Gronwall inequality gives $\phi_N(t) \leq c\norm{\widehat{K}}_{L^1(0,T)}(1+\norm{\widehat{r}}_{L^1(0,T)})$, a bound independent of $N$. This is the first estimate of the statement.

    For the replicas $Y^i$, the argument is identical, with the empirical average $\frac{1}{N}\sum_{j}\mu[Y^{j,N}_s]$ replaced by $\E[\mu[Y^i_s]]$ and $\alpha^{i,N}$ by $\alpha^i$; the quantity $\E[\norm{Y^i_s}_{L^1_\mu}^p]$, which is independent of $i$ by symmetry, satisfies the same inequality, giving the second estimate.

    Finally, since $X^{i,N}_t = g_0(t)+\mu[Y^{i,N}_t]$, we have $|X^{i,N}_t| \leq |g_0(t)| + \norm{Y^{i,N}_t}_{L^1_\mu}$, the boundedness of $g_0$ gives the corresponding bound for $X^{i,N}$; the same holds for $X^i$.
\end{proof}
In particular, since $|\alpha^{i,N}_t|\leq c\big(1+\norm{Y^{i,N}_t}_{L^1_\mu}\big)$, as shown in the proof above, the case $p=4$ gives $$\sup_{1\leq i\leq N}\sup_{t\in[0,T]}\E[|\alpha^{i,N}_t|^4]\leq c.$$ Hence the feedback profile $(\alpha^{i,N})_{1\leq i\leq N}$ is admissible, as required in Definition~\ref{def:nash}.
\begin{proposition}
\label{prop:propagation chaos}
    Suppose Assumptions~\ref{ass:kernel} and~\ref{ass:data} hold. Then, for every $1 \leq i \leq N$ and $p=2,3,4$, there exists a constant $c=c(p)>0$, independent of $N$, such that 
    \begin{equation}
        \sup_{t \in [0,T]} \E\left[ \norm{Y^{i,N}_t-Y^i_t}_{L^1_\mu}^p \right] \leq \frac{c}{N^{p/2}}.
    \end{equation}
    In particular, it holds
    \begin{equation}
    \label{eq:prop chaos X}
        \sup_{t \in [0,T]} \E\left[ |X^{i,N}_t-X^i_t|^p \right] \leq \frac{c}{N^{p/2}}.
    \end{equation}
\end{proposition}
\begin{proof}
    For any $1\leq i\leq N$, consider the difference
    \begin{align*}
         |Y^{i,N}_t(\theta)-Y^i_t(\theta)|  = \left| \int_0^t e^{-\theta(t-s)}\Delta^1_s\de s+ \int_0^t e^{-\theta(t-s)}\Delta^2_s\de W^i_s  \right|,
    \end{align*}
    where
    \begin{align*}
        \Delta^1_t &= B\left( \mu[Y^{i,N}_t]-\mu[Y^{i}_t]  \right) + B\left( \E\left[ \mu[Y^{i}_t] \right] - \frac{1}{N}\sum_{j=1}^N \mu[Y^{j,N}_t] \right) + C(\alpha^{i,N}_t - \alpha^i_t), \\
        \Delta^2_t &= D\left( \mu[Y^{i,N}_t]-\mu[Y^{i}_t]  \right) + D\left( \E\left[ \mu[Y^{i}_t] \right] - \frac{1}{N}\sum_{j=1}^N \mu[Y^{j,N}_t] \right) + F(\alpha^{i,N}_t - \alpha^i_t).
    \end{align*}

We focus on $\E[|\Delta^1_s|^p]$ since we treat the one with $\Delta^2$ in the same way
\begin{align*}
    \E\left[|\Delta^1_s|^p\right] \leq c\,\E\Bigg[ |B|^p \norm{Y^{i,N}_s - Y^i_s}_{L^1_\mu}^p + |C|^p |\alpha^{i,N}_s - \alpha^i_s|^p + |B|^p \left| \E\left[ \mu[Y^{i}_s] \right] - \frac{1}{N}\sum_{j=1}^N \mu[Y^{j,N}_s] \right|^p \Bigg].
\end{align*}
Now, by their definitions, it is easy to see that
\begin{equation}
\label{eq:Lip di alpha hat MF}
    |\alpha^{i,N}_s - \alpha^i_s| \leq c \norm{Y^{i,N}_s - Y^i_s}_{L^1_\mu}.
\end{equation}
Moreover
\begin{align*}
    \E\Bigg[ \left| \E\left[ \mu[Y^{i}_s] \right] - \frac{1}{N}\sum_{j=1}^N \mu[Y^{j,N}_s] \right|^p \Bigg] & \leq c\, \E\Bigg[ \left| \E\left[ \mu[Y^{i}_s] \right] - \frac{1}{N}\sum_{j=1}^N \mu[Y^{j}_s] \right|^p \Bigg] %
    + \E\Bigg[ \left| \frac{1}{N}\sum_{j=1}^N \left(\mu[Y^j_s]-\mu[Y^{j,N}_s]\right) \right|^p \Bigg].
\end{align*}
The first term is $O(N^{-p/2})$ by Lemma \ref{lemma:CLT}, while the second reduces to $\E[\norm{Y^{i,N}_s - Y^i_s}_{L^1_\mu}^p]$ since $\mu[Y^j_s]-\mu[Y^{j,N}_s]$ are identically distributed for every $1 \leq j \leq N$. We end up with
\begin{align*}
    \E\left[|\Delta^1_s|^p\right] \leq c \left( \mathbb{E}\!\left[ \|Y^{i,N}_s - Y^i_s\|_{L^1_\mu}^p \right] + \frac{1}{N^{p/2}} \right).
\end{align*}
Therefore, by Proposition \ref{prop:usual machinery} we get
\begin{equation*}
    \mathbb{E}\!\left[ \|Y^{i,N}_t - Y^i_t\|_{L^1_\mu}^p \right] \leq c\int_0^t \widehat{K}(t-s)\left( \mathbb{E}\!\left[ \|Y^{i,N}_s - Y^i_s\|_{L^1_\mu}^p \right] + \frac{1}{N^{p/2}} \right)\de s.
\end{equation*}
An application of convolutional Gronwall inequality gives
\begin{equation}
    \mathbb{E}\!\left[ \|Y^{i,N}_t - Y^i_t\|_{L^1_\mu}^p \right] \leq c\norm{\widehat{K}}_{L^1(0,T)}(1+\norm{\widehat{r}}_{L^1(0,T)}) \frac{1}{N^{p/2}},
\end{equation}
and we conclude by noticing that
$$
|X^{i,N}_t-X^i_t| \leq \|Y^{i,N}_t - Y^i_t\|_{L^1_\mu}, \quad t \in [0,T],
$$
in view of the projection map.
\end{proof}

\subsection{Approximated Nash equilibrium for the \texorpdfstring{$N$}{N}-agent system}
\label{ssec:nash}
We denote by $J^\infty$ the (optimal) cost of the mean field problem
\begin{equation}
    J^\infty(\alpha^i) = \E\left[ \int_0^T \left( f(X^i_t,\alpha^i_t) + Q\,\mathfrak{m}_t^2 + 2L\,\mathfrak{m}_t \right)\de t \right],
\end{equation}
where $X^i_t = g_0(t) + \mu[Y^i_t]$, we have used the consistency relation $\mathfrak{m}_t = \E[X^i_t]$, and we recall that $f(x,a)= Q x^2 + R a^2 + 2 Lx$ since we are working with $d=1$.
Analogously, we denote by $J^{i,N}$ the cost functional of the $i$-th player
\begin{equation}
    J^{i,N}(\alpha^{1,N},\dots,\alpha^{N,N}) := \E\left[ \int_0^T \left\{ f(X^{i,N}_t,\alpha^{i,N}_t) + Q\left(\frac{1}{N}\sum_{j=1}^N X^{j,N}_t\right)^2 + 2L\frac{1}{N}\sum_{j=1}^N X^{j,N}_t \right\}\de t \right].
\end{equation}
The first step compares, for a single player, the cost in the $N$-player game with its mean field counterpart, when all players use the equilibrium control.
\begin{lemma}
\label{lemma:comparison of J^infty and J^i,N on optimal alpha}
    Suppose Assumptions~\ref{ass:kernel} and~\ref{ass:data} hold. Then, for every $1 \leq i \leq N$, there exists a constant $c>0$, independent of $N$, such that
    \begin{equation}
        |J^\infty(\alpha^i)-J^{i,N}(\alpha^{1,N},\dots,\alpha^{N,N})| \leq \frac{c}{\sqrt N}.
    \end{equation}
\end{lemma}
\begin{proof}
    A direct computation leads to
    $$
    |J^\infty(\alpha^i)-J^{i,N}(\alpha^{1,N},\dots,\alpha^{N,N})| \leq c(
    \mathrm{\textbf{I}} + \mathrm{\textbf{II}} + \,\mathrm{\textbf{III}}),
    $$
    where
    \begin{align*}
        \mathrm{\textbf{I}} & =\E\left[ \int_0^T|f(X^i_t,\alpha^i_t)-f(X^{i,N}_t,\alpha^{i,N}_t)|\de t\right], \\[0.2cm]
        \mathrm{\textbf{II}} & = \E\left[ \int_0^T\left| \mathfrak{m}_t^2 -  \left(\frac{1}{N}\sum_{j=1}^N X^{j,N}_t\right)^2\right|\de t\right], \\[0.2cm]
        \mathrm{\textbf{III}} &=  \E\left[ \int_0^T\left| \mathfrak{m}_t - \frac{1}{N}\sum_{j=1}^N X^{j,N}_t \right|\de t\right].
    \end{align*}
    We start with the third.
    \begin{align*}
        \mathrm{\textbf{III}} \leq \E\left[ \int_0^T\left| \E[\mu[Y^i_t]] - \frac{1}{N}\sum_{j=1}^N \mu[Y^{j}_t] \right|\de t\right] + \E\left[ \int_0^T\left| \frac{1}{N}\sum_{j=1}^N (\mu[Y^{j}_t] - \mu[Y^{j,N}_t]) \right|\de t\right] \leq \frac{c}{\sqrt N},
    \end{align*}
    by Lemma \ref{lemma:CLT} and Proposition \ref{prop:propagation chaos}. Now,
    \begin{align*}
        \mathrm{\textbf{II}} = \E\left[ \int_0^T\left| \mathfrak{m}_t - \frac{1}{N}\sum_{j=1}^N X^{j,N}_t \right|\left| \mathfrak{m}_t + \frac{1}{N}\sum_{j=1}^N X^{j,N}_t \right|\de t\right] \leq \frac{c}{\sqrt N},
    \end{align*}
    by Cauchy-Schwarz inequality, the same computations done for $\mathrm{\textbf{III}}$ and the a priori bounds of Proposition~\ref{prop:a priori estimates}. At last,
    \begin{align*}
        \mathrm{\textbf{I}} \leq c\,(\mathrm{\textbf{I}}_a+\mathrm{\textbf{I}}_b+\mathrm{\textbf{I}}_c),
    \end{align*}
    where
    \begin{align*}
        \mathrm{\textbf{I}}_a  = \E\left[ \int_0^T |(X^{i}_t)^2-(X^{i,N}_t)^2|\de t\right],    \quad 
        \mathrm{\textbf{I}}_b  = \E\left[ \int_0^T |(\alpha^i_t)^2-(\alpha^{i,N}_t)^2|\de t\right], \quad
        \mathrm{\textbf{I}}_c  = \E\left[ \int_0^T |X^i_t-X^{i,N}_t|\de t\right].
    \end{align*}
    Now,
    \begin{align*}
        \mathrm{\textbf{I}}_c \leq \E\left[ \int_0^T \norm{Y^i_t-Y^{i,N}_t}_{L^1_\mu}\de t\right] \leq \frac{c}{\sqrt N},
    \end{align*}
    by Proposition \ref{prop:propagation chaos}. We deal with the other two terms in the same way after splitting the difference of the squares, using Cauchy-Schwarz inequality, the a priori bounds of Proposition~\ref{prop:a priori estimates} and, for $\mathrm{\textbf{I}}_b$, the estimate \eqref{eq:Lip di alpha hat MF}.
\end{proof}

Now, let $\eta$ be an admissible control and consider the dynamics of the players when player 1 (without loss of generality, we can work only on player 1 by symmetry) uses $\eta$ instead of $\alpha^{1,N}$, while the other players keep applying the equilibrium feedback to their own lifted states. We denote by $Y^{i,N,\eta}$ the lifted state of player $i$ after the deviation. More precisely, $Y^{1,N,\eta}$ and $Y^{i,N,\eta}$, $2\leq i\leq N$, solve
\begin{equation}
\label{eq:particles lifted with beta}
    \begin{aligned}
        Y^{1,N,\eta}_t(\theta) & = \int_0^t e^{-\theta(t-s)}\left[ \beta(s) + B \left( \mu[Y^{1,N,\eta}_s] - \frac{1}{N}\sum_{j=1}^N \mu[Y^{j,N,\eta}_s] \right) + C \eta_s\right]\de s\\
        & \qquad + \int_0^t e^{-\theta(t-s)}\left[ \gamma(s) + D \left( \mu[Y^{1,N,\eta}_s] - \frac{1}{N}\sum_{j=1}^N \mu[Y^{j,N,\eta}_s] \right) + F \eta_s\right]\de W_s^1, \\[0.2cm]
        Y^{i,N,\eta}_t(\theta) & = \int_0^t e^{-\theta(t-s)}\left[ \beta(s) + B \left( \mu[Y^{i,N,\eta}_s] - \frac{1}{N}\sum_{j=1}^N \mu[Y^{j,N,\eta}_s] \right) + C \alpha^{i,N,\eta}_s\right]\de s\\
        & \qquad + \int_0^t e^{-\theta(t-s)}\left[ \gamma(s) + D \left( \mu[Y^{i,N,\eta}_s] - \frac{1}{N}\sum_{j=1}^N \mu[Y^{j,N,\eta}_s] \right) + F \alpha^{i,N,\eta}_s\right]\de W_s^i,
    \end{aligned}
\end{equation}
where, for $2\leq i\leq N$,
$$
\alpha^{i,N,\eta}_t := -\hat{R}(\Gamma_t)^{-1} \left\{ h(t,\Gamma_t,\Lambda_t^\mathfrak{m},\mathfrak{m}) + \int_0^\infty S(\Gamma_t)(\theta)Y^{i,N,\eta}_t(\theta)\mu(\de\theta) \right\}.
$$
Since this feedback is affine in the lifted state with bounded coefficients, the system is well-posed by \cite[Theorem 4.1]{LQ_Volterra_AJMP}, and, as in \eqref{eq:Lip di alpha hat MF},
\begin{equation}
\label{eq:non-deviating controls}
    |\alpha^{i,N,\eta}_t| \leq c\big(1+\norm{Y^{i,N,\eta}_t}_{L^1_\mu}\big), \qquad |\alpha^{i,N,\eta}_t-\alpha^{i,N}_t| \leq c\,\norm{Y^{i,N,\eta}_t-Y^{i,N}_t}_{L^1_\mu}, \qquad 2\leq i\leq N.
\end{equation}
These are the only properties of the controls of the non-deviating players used below.

\medskip 

The a priori bounds of Proposition~\ref{prop:a priori estimates} extend to the deviated system, at the price of a dependence on the size of the deviation.
From here on the estimates involving $\eta$ are stated in time-integrated form, rather than uniformly in $t$ as in Section~\ref{ssec:convergence}, since the integral norm is the one controlled by the cost, as shown in Remark~\ref{rmk:competitive}.
\begin{proposition}
\label{prop:a priori estimates deviated}
    Suppose Assumptions~\ref{ass:kernel} and~\ref{ass:data} hold. Then, for every $p=2,3,4$ there exists a constant $c=c(p)>0$, independent of $N$, such that for every admissible deviation $\eta \in \mathcal{A}$ of player $1$,
    \begin{equation}
    \label{eq:a priori deviated}
        \sup_{1\leq i\leq N}\int_0^T\E\left[\norm{Y^{i,N,\eta}_t}_{L^1_\mu}^p\right]\de t \leq c\left(1 + \int_0^T\E\left[|\eta_t|^p\right]\de t\right),
    \end{equation}
    and the same bound holds for $X^{i,N,\eta}$ in place of $Y^{i,N,\eta}$.
\end{proposition}
\begin{proof}
    The argument is the one of Proposition~\ref{prop:a priori estimates}, with two modifications. For the players $i\geq 2$, the first estimate in \eqref{eq:non-deviating controls} gives $\E[|\alpha^{i,N,\eta}_s|^p]\leq c\big(1+\E[\norm{Y^{i,N,\eta}_s}_{L^1_\mu}^p]\big)$, so their coefficients are estimated as before; for player~$1$ the exogenous control $\eta$ replaces $\alpha^{1,N}$ and contributes the additional terms $|C|^p\,\E[|\eta_s|^p]$ and $|F|^p\,\E[|\eta_s|^p]$. Writing $\varphi^\eta_N(s):=\sup_{1\leq i\leq N}\E[\norm{Y^{i,N,\eta}_s}_{L^1_\mu}^p]$, which is finite for each fixed $N$ and $\eta$, and keeping the contribution of $\eta$ inside the convolution, Proposition~\ref{prop:usual machinery} gives
    $$
    \varphi^\eta_N(t) \leq c\int_0^t \widehat{K}(t-s)\Big(1 + \varphi^\eta_N(s) + \E[|\eta_s|^p]\Big)\de s.
    $$
    By the convolutional Gronwall inequality, denoting by $\widehat{r}$ the resolvent of $-c\widehat{K}$, which is non-positive, we get
    $$
    \varphi^\eta_N(t) \leq c\int_0^t \big(-\widehat{r}(t-s)\big)\Big(1+\E[|\eta_s|^p]\Big)\de s.
    $$
    By Young's convolution inequality in $L^1(0,T)$, integrating in $t$ over $[0,T]$ yields
    $$
    \int_0^T \varphi^\eta_N(t)\,\de t \leq c\left(1+\int_0^T \E[|\eta_s|^p]\de s\right),
    $$
    the constant depending on the data of the problem only through the $L^1(0,T)$-norm of $\widehat{r}$, and in particular not on $N$. This entails \eqref{eq:a priori deviated}. The bound for $X^{i,N,\eta}$ follows from $|X^{i,N,\eta}_t| \leq |g_0(t)| + \norm{Y^{i,N,\eta}_t}_{L^1_\mu}$ and the boundedness of $g_0$.
\end{proof}

We now quantify the effect of the deviation of player~$1$ on the whole system: since a single player influences the empirical mean only through a factor $1/N$, all the other players remain close to their equilibrium counterparts.

\begin{lemma}
    Suppose Assumptions~\ref{ass:kernel} and~\ref{ass:data} hold. Then, for every admissible deviation $\eta$ of player $1$, there exists a constant $c>0$, independent of $N$, such that
    \[
    \frac{1}{N} \sum_{j=1}^N \int_0^T \E\left[ \norm{Y^{j,N,\eta}_t-Y^{j,N}_t}_{L^1_\mu}^2 \right]\de t \leq \frac{c}{N} \int_0^T \E\left[ |\eta_t-\alpha^{1,N}_t|^2 \right]\de t.
    \]
\end{lemma}
\begin{proof}
    Fix $2 \leq i \leq N$. By Proposition~\ref{prop:usual machinery} and the second estimate in \eqref{eq:non-deviating controls}, we have
    \begin{align*}
        &\E\left[ \norm{Y^{i,N,\eta}_t-Y^{i,N}_t}_{L^1_\mu}^2 \right]
        \leq c \int_0^t \widehat{K}(t-s) \left( \E\left[ \norm{Y^{i,N,\eta}_s-Y^{i,N}_s}_{L^1_\mu}^2\right]
        + \frac{1}{N}\sum_{j=1}^N \E\left[\norm{Y^{j,N,\eta}_s-Y^{j,N}_s}_{L^1_\mu}^2 \right]\right)\de s.
    \end{align*}
    For player $1$, the same argument yields
    \begin{align*}
    \E\left[ \norm{Y^{1,N,\eta}_t-Y^{1,N}_t}_{L^1_\mu}^2 \right]
        &\leq c \int_0^t \widehat{K}(t-s) \Bigg( \E\left[\norm{Y^{1,N,\eta}_s-Y^{1,N}_s}_{L^1_\mu}^2 \right]
        \\
        &\quad+ \frac{1}{N}\sum_{j=1}^N\E\left[\norm{Y^{j,N,\eta}_s-Y^{j,N}_s}_{L^1_\mu}^2 \right]
        + \E\left[ |\eta_s-\alpha^{1,N}_s|^2 \right] \Bigg)\de s,
    \end{align*}
    the contribution of the deviation being kept inside the convolution. Averaging over all players $i = 1, \ldots, N$ and setting $\phi_N(t) := \frac{1}{N}\sum_{i=1}^N \E[ \norm{Y^{i,N,\eta}_t- Y^{i,N}_t}_{L^1_\mu}^2 ]$, we obtain
    \begin{align*}
        \phi_N(t) \leq c \int_0^t \widehat{K}(t-s)\left( \phi_N(s)
        + \frac{1}{N} \E\left[ |\eta_s-\alpha^{1,N}_s|^2 \right] \right)\de s.
    \end{align*}
    The convolutional Gronwall inequality, followed by Young's convolution inequality in $L^1(0,T)$, gives the conclusion.
\end{proof}
A consequence of the previous lemma and of its proof is that, for $2 \leq i \leq N$,
\begin{equation}
    \int_0^T \E\left[ \norm{Y^{i,N,\eta}_t-Y^{i,N}_t}_{L^1_\mu}^2 \right]\de t \leq \frac{c}{N} \int_0^T \E\left[ |\eta_t-\alpha^{1,N}_t|^2 \right]\de t.
\end{equation}
Therefore, we can control the distance between $Y^{i,N,\eta}$ and $Y^i$ by
\begin{equation*}
    \begin{aligned}
        \int_0^T \E\left[ \norm{Y^{i,N,\eta}_t-Y^{i}_t}_{L^1_\mu}^2 \right]\de t &\leq c\,\left( \int_0^T \E\left[ \norm{Y^{i,N,\eta}_t-Y^{i,N}_t}_{L^1_\mu}^2 \right]\de t +  \int_0^T \E\left[ \norm{Y^{i,N}_t-Y^{i}_t}_{L^1_\mu}^2 \right]\de t \right) \\
        & \leq \frac{c}{N}\left( \int_0^T \E\left[ |\eta_t-\alpha^{1,N}_t|^2 \right]\de t +1\right), \qquad 2\leq i \leq N,
    \end{aligned}
\end{equation*}
where the second term has been estimated by Proposition~\ref{prop:propagation chaos}. We say that a deviation $\eta$ of player $1$ is \emph{competitive} if
\begin{equation}
\label{eq:competitive}
    \int_0^T \E\left[ |\eta_t-\alpha^{1,N}_t|^2 \right]\de t \leq \kappa
\end{equation}
for some constant $\kappa$ independent of $N$; as shown in Remark~\ref{rmk:competitive} below, it is enough to establish the approximate Nash property for such deviations. For a competitive deviation we thus obtain
\begin{equation}
\label{eq:estimates i,N,eta vs i}
    \int_0^T \E\left[ \norm{Y^{i,N,\eta}_t-Y^{i}_t}_{L^1_\mu}^2 \right]\de t \leq \frac{c}{N}, \qquad 2\leq i\leq N,
\end{equation}
so that
\begin{equation}
    \frac{1}{N-1} \sum_{j=2}^N  \int_0^T\E\left[ \norm{Y^{j,N,\eta}_t-Y^{j}_t}_{L^1_\mu}^2 \right]\de t \leq \frac{c}{N}.
\end{equation}

As a consequence, the empirical mean of the deviated system still converges to the equilibrium curve $\mathfrak{m}$, with the same rate as in the undeviated case.

\begin{lemma}
\label{lemma:approximation of fixed point with eta particles}
    Suppose Assumptions~\ref{ass:kernel} and~\ref{ass:data} hold. Then, for every competitive admissible deviation $\eta$ of player $1$, there exists a constant $c>0$, independent of $N$, such that
    \[
    \int_0^T \E\left[ \left| \frac{1}{N}\sum_{j=1}^N X^{j,N,\eta}_t -\mathfrak{m}_t \right|^2  \right]\de t \leq \frac{c}{N},
    \]
    where $X^{j,N,\eta}_t=g_0(t) + \mu[Y^{j,N,\eta}_t]$.
\end{lemma}
\begin{proof}
    We can write
    $$
    \int_0^T \E\left[ \left| \frac{1}{N}\sum_{j=1}^N X^{j,N,\eta}_t -\mathfrak{m}_t \right|^2  \right]\de t \leq c\,(\mathrm{\textbf{I}} + \mathrm{\textbf{II}} + \mathrm{\textbf{III}}  ),
    $$
    where
    \begin{align*}
        \mathrm{\textbf{I}} & = \int_0^T\E\left[ \left| \frac{1}{N}\sum_{j=1}^N X^{j,N,\eta}_t -\frac{1}{N-1}\sum_{j=2}^N X^{j,N,\eta}_t  \right|^2  \right]\de t, \\[0.2cm]
        \mathrm{\textbf{II}} & = \int_0^T\E\left[ \left| \frac{1}{N-1}\sum_{j=2}^N X^{j,N,\eta}_t -\frac{1}{N-1}\sum_{j=2}^N X^{j}_t  \right|^2  \right]\de t, \\[0.2cm]
        \mathrm{\textbf{III}} & = \int_0^T\E\left[ \left| \frac{1}{N-1}\sum_{j=2}^N X^{j}_t -\mathfrak{m}_t  \right|^2  \right]\de t.
    \end{align*}
    For the first term,
    \begin{align*}
        \mathrm{\textbf{I}} & = \int_0^T\E\left[ \left| \frac{1}{N} \left\{  \frac{1}{N-1}\sum_{j=2}^N (X^{1,N,\eta}_t-X^{j,N,\eta}_t) \right\} \right|^2  \right]\de t \\[0.2cm]
        & \leq \frac{2}{N^2(N-1)} \sum_{j=2}^N \int_0^T\E\left[ |X^{1,N,\eta}_t|^2+|X^{j,N,\eta}_t|^2 \right]\de t \leq \frac{c}{N^2} \leq \frac{c}{N},
    \end{align*}
    by Proposition~\ref{prop:a priori estimates deviated}, since $\eta$ is competitive and therefore $\int_0^T\E[|\eta_t|^2]\de t$ is bounded uniformly in $N$. For the second term,
    \begin{align*}
        \mathrm{\textbf{II}} & \leq \frac{1}{N-1} \sum_{j=2}^N \int_0^T\E\left[ \norm{Y^{j,N,\eta}_t-Y^j_t}^2_{L^1_\mu} \right]\de t \leq\frac{c}{N},
    \end{align*}
    by \eqref{eq:estimates i,N,eta vs i}. Lastly, the processes $\{X^j\}_{2\leq j\leq N}$ are i.i.d.\ with $\E[X^j_t]=\mathfrak{m}_t$ by the consistency condition, and have bounded second moments by Proposition~\ref{prop:a priori estimates}, so that Lemma~\ref{lemma:CLT} gives
    \begin{align*}
        \mathrm{\textbf{III}} \leq \frac{cT}{N-1}.
    \end{align*}
    Since $N-1\geq N/2$ for $N\geq 2$, combining the three contributions yields the thesis.
\end{proof}

Let us define the following process
\begin{align*}
X^{1,\eta}_t &:= g_0(t) + \int_0^t K(t-s)[\beta(s) %
+ B(X^{1,\eta}_s-\mathfrak{m}_s) + C\eta_s]\de s+ \int_0^t K(t-s)[\gamma(s) + D(X^{1,\eta}_s-\mathfrak{m}_s) + F\eta_s]\de W^{1}_s,
\end{align*}
and let us introduce the corresponding lifted process $Y^{1,\eta}$ with dynamics
\begin{align*}
    Y_t^{1,\eta}(\theta) &= \int_0^t e^{-\theta(t-s)}[\beta(s)+Bg_0(s)+B(\mu[Y_s^{1,\eta}]-\mathfrak{m}_s)+C\eta_s]\de s\\
    &\qquad+ \int_0^t e^{-\theta(t-s)}[\gamma(s) +Dg_0(s) + D(\mu[Y_s^{1,\eta}]-\mathfrak{m}_s) + F\eta_s]\de W_s^1.
\end{align*}
We can then transfer the propagation of chaos to the deviating player itself.
\begin{lemma}
\label{lemma:prop of chaos on player 1 when he plays eta}
    Suppose Assumptions~\ref{ass:kernel} and~\ref{ass:data} hold. Then, for every competitive deviation $\eta$ of player $1$, there exists a constant $c>0$, independent of $N$, such that
    \begin{equation*}
        \int_0^T \E\left[ \norm{Y^{1,N,\eta}_t-Y^{1,\eta}_t}_{L^1_\mu}^2 \right]\de t \leq \frac{c}{N}.
    \end{equation*}
\end{lemma}
\begin{proof}
    By Proposition \ref{prop:usual machinery}, we get
    $$
\mathbb{E}\!\left[ \|Y^{1,N,\eta}_t - Y^{1,\eta}_t\|_{L^1_\mu}^2 \right] \leq c\int_0^t \widehat{K}(t-s)\E\left[|\Delta^1_s|^2+|\Delta^2_s|^2\right]\de s,
$$
where
\begin{align*}
    \Delta^1_s & = B\left(\mu[Y^{1,N,\eta}_s]-\mu[Y^{1,\eta}_s]\right) + B\left( \mathfrak{m}_s-\frac{1}{N}\sum_{j=1}^N \mu[Y^{j,N,\eta}_s] \right) -Bg_0(s),\\
    \Delta^2_s & = D\left(\mu[Y^{1,N,\eta}_s]-\mu[Y^{1,\eta}_s]\right) + D\left( \mathfrak{m}_s-\frac{1}{N}\sum_{j=1}^N \mu[Y^{j,N,\eta}_s] \right) -Dg_0(s).
\end{align*}
Recalling that $\mu[Y^{j,N,\eta}_s]=X^{j,N,\eta}_s-g_0(s)$, we obtain
\[
\mathbb{E}\!\left[ \|Y^{1,N,\eta}_t - Y^{1,\eta}_t\|_{L^1_\mu}^2 \right] \leq c\int_0^t \widehat{K}(t-s)\left(\mathbb{E}\!\left[ \|Y^{1,N,\eta}_s - Y^{1,\eta}_s\|_{L^1_\mu}^2 \right] + \E\left[\Big|\frac{1}{N}\sum_{j=1}^N X^{j,N,\eta}_s - \mathfrak{m}_s\Big|^2\right]\right)\de s,
\]
and the conclusion follows by the convolutional Gronwall inequality, Young's convolution inequality in $L^1(0,T)$, and Lemma \ref{lemma:approximation of fixed point with eta particles}.
\end{proof}

As a last step towards the $\eps$-Nash property, we denote by
\begin{equation}
    J^{1,N}(\eta,\dots,\alpha^{N,N}) := \E \int_0^T \left\{ f(X^{1,N,\eta}_t,\eta_t) + Q\left(\frac{1}{N}\sum_{j=1}^N X^{j,N,\eta}_t\right)^2 + \frac{2L}{N}\sum_{j=1}^N X^{j,N,\eta}_t \right\}\de t,
\end{equation}
the cost functional of player 1 when he plays strategy $\eta$ and by
\begin{equation}
    J^\infty(\eta) = \E\left[ \int_0^T \left( f(X^{1,\eta}_t,\eta_t) + Q\,\mathfrak{m}_t^2 + 2L\,\mathfrak{m}_t \right)\de t\right],
\end{equation}
the cost functional of the mean field problem when player 1 plays $\eta$.

\smallskip

With these estimates at hand, we can compare the two costs also when player~$1$ deviates, which is the second and last ingredient of the argument.
\begin{lemma}
    Suppose Assumptions~\ref{ass:kernel} and~\ref{ass:data} hold. Then, for every competitive deviation $\eta$ of player $1$, there exists a constant $c>0$, independent of $N$, such that
    \begin{equation}
        |J^{1,N}(\eta,\alpha^{2,N},\dots,\alpha^{N,N})-J^\infty(\eta)|\leq \frac{c}{\sqrt{N}}.
    \end{equation}
\end{lemma}
\begin{proof}
    The proof is, mutatis mutandis, the same as the one in Lemma \ref{lemma:comparison of J^infty and J^i,N on optimal alpha} using Lemma \ref{lemma:approximation of fixed point with eta particles} and Lemma \ref{lemma:prop of chaos on player 1 when he plays eta}. Indeed,
    $$
    |J^\infty(\eta)-J^{1,N}(\eta,\dots,\alpha^{N,N})| \leq c(
    \mathrm{\textbf{I}} + \mathrm{\textbf{II}} + \,\mathrm{\textbf{III}}),
    $$
    where
    \begin{align*}
        \mathrm{\textbf{I}} & =\E\left[ \int_0^T\left|f(X^{1,\eta}_t,\eta_t)-f(X^{1,N,\eta}_t,\eta_t)\right|\de t\right], \\[0.2cm]
        \mathrm{\textbf{II}} & = \E\left[ \int_0^T\left| \mathfrak{m}_t^2 -  \left(\frac{1}{N}\sum_{j=1}^N X^{j,N,\eta}_t\right)^2\right|\de t\right], \\[0.2cm]
        \mathrm{\textbf{III}} &=  \E\left[ \int_0^T\left| \mathfrak{m}_t - \frac{1}{N}\sum_{j=1}^N X^{j,N,\eta}_t \right|\de t\right].
    \end{align*}
    We start with the third, for which it holds
    \begin{align*}
        \mathrm{\textbf{III}}  \leq \frac{c}{\sqrt N},
    \end{align*}
    by the Cauchy-Schwarz inequality and Lemma \ref{lemma:approximation of fixed point with eta particles}. Now,
    \begin{align*}
        \mathrm{\textbf{II}} = \E\left[ \int_0^T\left| \mathfrak{m}_t - \frac{1}{N}\sum_{j=1}^N X^{j,N,\eta}_t \right|\left| \mathfrak{m}_t + \frac{1}{N}\sum_{j=1}^N X^{j,N,\eta}_t \right|\de t\right] \leq \frac{c}{\sqrt N},
    \end{align*}
    by Cauchy-Schwarz inequality, Lemma \ref{lemma:approximation of fixed point with eta particles} and the a priori bounds of Proposition~\ref{prop:a priori estimates deviated}. At last,
    \begin{align*}
        \mathrm{\textbf{I}} \leq c\,(\mathrm{\textbf{I}}_a+\mathrm{\textbf{I}}_b),
    \end{align*}
    where
    \begin{align*}
        \mathrm{\textbf{I}}_a  = \E\left[ \int_0^T |(X^{1,\eta}_t)^2-(X^{1,N,\eta}_t)^2|\de t\right], \qquad 
        \mathrm{\textbf{I}}_b  = \E\left[ \int_0^T |X^{1,\eta}_t-X^{1,N,\eta}_t|\de t\right].
    \end{align*}
    Now,
    \begin{align*}
        \mathrm{\textbf{I}}_b \leq \E\left[ \int_0^T \norm{Y^{1,\eta}_t-Y^{1,N,\eta}_t}_{L^1_\mu}\de t\right] \leq \frac{c}{\sqrt N},
    \end{align*}
    by the Cauchy-Schwarz inequality and Lemma \ref{lemma:prop of chaos on player 1 when he plays eta}. We deal with the other term in the same way after splitting the difference of the squares, using Cauchy-Schwarz inequality and the a priori bounds of Proposition~\ref{prop:a priori estimates deviated}, together with the analogous bound for $X^{1,\eta}$, which follows from the same argument.
\end{proof}

\begin{remark}
\label{rmk:competitive}
    The approximate Nash property need only be established for competitive deviations, in the sense of \eqref{eq:competitive}. Indeed, set $u:=\E\int_0^T|\eta_t|^2\de t$. Since $Q \in S^d_+$ and $R-\lambda I_q \in S^q_+$ by Assumption~\ref{ass:data}(iii), both the running cost and the interaction term are bounded below by their linear parts, so that
    $$
    J^{1,N}(\eta,\dots,\alpha^{N,N}) \geq \lambda\, u - 2|L|\,\E\!\int_0^T\left( |X^{1,N,\eta}_t| + \Big|\frac{1}{N}\sum_{j=1}^N X^{j,N,\eta}_t\Big| \right)\de t \geq \frac{\lambda}{2}\,u - c,
    $$
    where the last inequality follows from the Cauchy-Schwarz inequality and Proposition~\ref{prop:a priori estimates deviated}, which bound the two expectations by $c(1+\sqrt{u})$, the term of order $\sqrt{u}$ being then absorbed into $\lambda u$ by Young's inequality; the constant $c$ is independent of $N$. Since $J^{1,N}(\alpha^{1,N},\dots,\alpha^{N,N})$ is bounded above uniformly in $N$ by Proposition~\ref{prop:a priori estimates}, while the lower bound above grows linearly in $u$, there exists a threshold $u_0$, independent of $N$, such that any deviation with $u > u_0$ satisfies $J^{1,N}(\eta,\dots,\alpha^{N,N}) \geq J^{1,N}(\alpha^{1,N},\dots,\alpha^{N,N})$, and the corresponding $\eps_N$-Nash inequality holds trivially. It is therefore enough to consider deviations with $u \leq u_0$; for these,
    $$
    \int_0^T \E\left[ |\eta_t-\alpha^{1,N}_t|^2 \right]\de t \leq 2u_0 + 2\int_0^T \E\left[ |\alpha^{1,N}_t|^2 \right]\de t \leq \kappa
    $$
    for some $\kappa$ independent of $N$, the last bound following again from Proposition~\ref{prop:a priori estimates} and the affine form of the equilibrium feedback. Hence such deviations are competitive. This is the argument of \cite[Section~4]{CarmonaDelarue2013}, adapted to the present setting.
\end{remark}
We can now conclude. Set $\eps_N := c/\sqrt{N}$, where $c$ is the larger of the two constants appearing in the comparison lemmas above. For a competitive deviation $\eta$ of player~$1$,
\begin{align*}
    J^{1,N}(\alpha^{1,N},\dots,\alpha^{N,N}) \leq J^\infty(\alpha^1) + \eps_N \leq J^\infty(\eta) + \eps_N \leq J^{1,N}(\eta,\dots,\alpha^{N,N}) + 2\eps_N,
\end{align*}
where the first and third inequalities are the two comparison lemmas, and the second is the optimality of the mean field control for $J^\infty$ (Proposition~\ref{prop:best-response}). By symmetry, the same holds for every player, and the claim follows with $2c$ in place of $c$. We have thus proved the following:

\epsnash*
\begin{remark}
\label{rmk:open loop}
    Theorem~\ref{thm:eps-nash} also holds in the open loop sense of \cite[Section~2.1.2]{CarmonaDelarue}, in which the players $j\neq i$ keep the processes $\alpha^{j,N}$ when player $i$ deviates. In that case $\alpha^{i,N,\eta}=\alpha^{i,N}$ for $i\geq2$: the second estimate in \eqref{eq:non-deviating controls} holds trivially, and the first one is replaced by $|\alpha^{i,N,\eta}_t|\leq c\big(1+\norm{Y^{i,N}_t}_{L^1_\mu}\big)$, whose moments are bounded by Proposition~\ref{prop:a priori estimates}. With these changes, the proofs of this section apply verbatim.
\end{remark}

\appendix
\section{Tools for Volterra Kernels}
\label{appendix:volterra tools}
This appendix collects classical results on scalar convolution Volterra integral equations, following the treatment of
\cite{Gripenberg-Londed-Staffans}.
All kernels are assumed to be of \emph{convolution type}, i.e.\ they depend only
on the difference of their arguments, and the underlying interval is
$\mathbb{R}_{+}=[0,+\infty)$.

We consider the convolution Volterra integral equation
\begin{equation}
  \label{eq:VIE}
  x(t) + \int_{0}^{t} K(t-s)\, x(s)\,\mathrm{d}s= f(t), \quad t \geq 0,
\end{equation}
where $K:\mathbb{R}_{+}\to\mathbb{R}$ is a locally integrable kernel and
$f:\mathbb{R}_{+}\to\mathbb{R}$ is a given forcing function.
Throughout this appendix we write $K*x$ for the convolution
$(K*x)(t)=\int_{0}^{t}K(t-s)x(s)\,\de s$.

The central object in the theory of convolution Volterra equations is the
\emph{resolvent}, which encodes the solution operator for Equation
\eqref{eq:VIE} independently of the forcing term $f$.
 
\begin{definition}[Resolvent kernel]
\label{def:resolvent}
Let $K\in L^{1}_{\mathrm{loc}}(\mathbb{R}_{+})$.
A function $r\in L^{1}_{\mathrm{loc}}(\mathbb{R}_{+})$ is called the
\emph{resolvent kernel} of $K$ if it satisfies the
\emph{resolvent equation}
\begin{equation}
  \label{eq:resolvent}
  r(t) + \int_{0}^{t} K(t-s)\,r(s)\,\mathrm{d}s= r(t) + \int_{0}^{t} r(t-s)\,K(s)\,\mathrm{d}s = K(t),
       \qquad t \geq 0.
\end{equation}
Equivalently, in convolution notation, $K - r = K*r =  r*K$.
\end{definition}
The resolvent kernel always exists and is unique in
$L^{1}_{\mathrm{loc}}(\mathbb{R}_{+})$ for every locally integrable kernel $K$(see \cite[Theorem~2.3.1]{Gripenberg-Londed-Staffans}). Its importance lies in the following representation theorem.

\begin{theorem}
\label{thm:variation_constants}
Let $K\in L^{1}_{\mathrm{loc}}(\mathbb{R}_{+})$ and let $r$ be its resolvent kernel. For any $f\in L^{1}_{\mathrm{loc}}(\mathbb{R}_{+})$, the unique solution of
\eqref{eq:VIE} in $L^{1}_{\mathrm{loc}}(\mathbb{R}_{+})$ is given by
\begin{equation}
  \label{eq:var_constants}
  x(t) = f(t) - \int_{0}^{t} r(t-s)\,f(s)\,\mathrm{d}s, \qquad t\geq 0.
\end{equation}
\end{theorem}
A class of kernels that plays a distinguished role both analytically and in
applications is that of \emph{completely monotone} functions.

\begin{definition}[Completely monotone function]
\label{def:CM}
A function $K:(0,+\infty)\to\mathbb{R}$ is called
\emph{completely monotone} if it is of class
$\rmC^{\infty}(0,+\infty)$ and satisfies
\begin{equation}
  \label{eq:CM}
  (-1)^{n}\, \frac{\d^n}{\d t^n}K(t) \;\geq\; 0
  \qquad \text{for all } t>0 \text{ and all } n\in\mathbb{N}.
\end{equation}
\end{definition}
In other words, $K$ is non-negative, non-increasing, convex, and all its
higher derivatives alternate in sign.
The fundamental characterisation of completely monotone functions is given by the Bernstein theorem.
\begin{theorem}
\label{thm:BW}
A function $K:(0,+\infty)\to\mathbb{R}$ is completely monotone if and only if it
can be represented as the Laplace transform of a non-negative Borel measure
$\mu$ on $[0,+\infty)$:
\begin{equation}
  \label{eq:BW}
  K(t) \;=\; \int_0^\infty e^{-t\theta}\,\mu(\mathrm{d}\theta),
  \qquad t>0.
\end{equation}
\end{theorem}
This is precisely the representation postulated in Assumption~\ref{ass:kernel}: the measure $\mu$ generates the Markovian lifting of the Volterra dynamics used throughout the paper.

Completely monotone kernels enjoy strong stability properties.
In particular, the class is closed under positive linear combinations and
products.

\begin{proposition}
\label{prop:CM_closure}
Let $a,b:(0,+\infty)\to\mathbb{R}$ be completely monotone.
\begin{enumerate}
  \item For any $\alpha,\beta\geq 0$, the function $\alpha a + \beta b$ is
        completely monotone.
  \item The pointwise product $a\cdot b$ is completely monotone.
\end{enumerate}
\end{proposition}

The following proposition, which is \cite[Lemma 2.1]{Hamaguchi}, gives a connection between the properties of a completely monotone kernel $K$ and the properties of its associated Bernstein's measure $\mu$.

\begin{proposition}
\label{prop:mu_properties}
    Let $K:(0,\infty)\to [0,\infty)$ be a completely monotone kernel and let $\mu$ be the corresponding measure given by Bernstein's theorem. Then:
    \begin{enumerate}
        \item $K$ is singular (i.e. $\lim_{t \to 0^+}K(t) = + \infty$) if and only if $\mu([1,+\infty))=+\infty$;
    \item Let $\alpha \in (0,1]$. Then, $\int_0^1 t^{\alpha-1}K(t)\de t < \infty$ if and only if $\int_1^\infty \theta^{-\alpha}\mu(\de\theta)<\infty$. Moreover, these conditions imply $\int_0^1 K(t)^{1/\alpha}\de t < \infty$.
    \end{enumerate}
\end{proposition}

Following the exposition in \cite[Section 9]{Gripenberg-Londed-Staffans}, we now recall the \emph{convolutional Gronwall inequality}, which is used repeatedly in the paper. Its statement requires the following notion.

\begin{definition}
    Let $K\colon\real_+\to\real$ be a scalar kernel, extended by $K(t)=0$ for $t<0$, let $J \subset \real_+$ be an interval and let $p,q \in [1,+\infty]$ be conjugate exponents. We say that $K$ is of type $L^p$ on $J$ if
    \begin{equation}
        \trinorm{K}_{L^p(J)} := \sup \int_J \int_J |g(t)K(t-s)f(s)|\de s \, \de t < \infty,
    \end{equation}
    where the supremum is taken over all $g \in L^q(J)$ and $f \in L^p(J)$ with $\norm{g}_{L^q(J)}\leq 1$ and $\norm{f}_{L^p(J)}\leq1$.
\end{definition}

\begin{theorem}[Generalized Gronwall Lemma]
\label{thm:gronwall}
    Let $K:\real_+ \to \real$ be a scalar kernel of type $L^p$ on an interval $J \subset \real_+$ for some $p \in [1,\infty]$ and assume that $-K$ has a non-positive resolvent $r$ of type $L^p$. Let $x,f \in L^p_{\mathrm{loc}}(J)$ and suppose
    $$
    x(t) \leq f(t) + \int_0^t K(t-s)x(s) \de s, \quad \textnormal{for a.e. }t \in J. 
    $$
    Then,
    \begin{equation}
        x(t) \leq f(t) -\int_0^t r(t-s)f(s)\de s, \quad \textnormal{for a.e. }t \in J.
    \end{equation}
\end{theorem}

The following proposition gives sufficient conditions to verify its assumptions.
\begin{proposition}
    Let $K:\real_+ \to \real$ be a scalar non-positive kernel of type $L^p$ on a compact interval $J \subset \real_+$ for some $p \in [1,\infty]$. Suppose that $J$ can be divided into finitely many subintervals $J_i$ such that on each $J_i$, $\trinorm{K}_{L^p(J_i)}<1$. Then, the resolvent $r$ of $K$ is of type $L^p$ on $J$ and it is non-positive.
\end{proposition}

We conclude with two properties used in Section~\ref{subsec:arbitrary horizons}, for the energy identity and for the Fredholm alternative respectively.
\begin{proposition}
\label{prop:positive definite}
    Let $K$ be a completely monotone kernel with Bernstein measure $\mu$ such that $K\in L^1(0,T)$. Then, for every $f\in L^2(0,T;\real^d)$,
    $$
    \int_0^T f_t^\top(K*f)_t\,\de t\geq0 .
    $$
\end{proposition}
\begin{proof}
    For $\theta\geq0$ let $F^\theta_t:=\int_0^te^{-\theta(t-s)}f_s\de s$. This function is absolutely continuous, with $F^\theta_0=0$ and $\frac{\de}{\de t}F^\theta_t=-\theta F^\theta_t+f_t$ for a.e.\ $t$. Hence
    $$
    \int_0^Tf_t^\top F^\theta_t\de t=\int_0^T\Big(\frac{\de}{\de t}F^\theta_t+\theta F^\theta_t\Big)^\top F^\theta_t\de t=\frac12|F^\theta_T|^2+\theta\int_0^T|F^\theta_t|^2\de t\geq0 .
    $$
    Since $\int_0^T\int_0^tK(t-s)|f_t||f_s|\de s\de t\leq\norm{K}_{L^1(0,T)}\norm{f}^2_{L^2(0,T)}<\infty$ by Young's inequality, an exchange of integrals gives
    $$
    \int_0^Tf_t^\top(K*f)_t\de t=\int_0^\infty\Big(\int_0^Tf_t^\top F^\theta_t\de t\Big)\mu(\de\theta)\geq0 .
    $$
\end{proof}

\begin{lemma}
\label{lemma:V compact}
    Suppose Assumption~\ref{ass:kernel} holds, and denote by $B_b([0,T];\real^d)$ the space of bounded measurable functions endowed with the sup norm. For $\varphi\in B_b([0,T];\real^d)$ set $(\mathcal V\varphi)_t(\theta):=\int_0^te^{-\theta(t-s)}\varphi_s\de s$. Then $\mathcal V$ is a compact linear operator from $B_b([0,T];\real^d)$ to $\rmC([0,T];L^1_\mu)$.
\end{lemma}
\begin{proof}
    Recall that a bounded linear operator is compact if it maps bounded sets into sets with compact closure. Since $\rmC([0,T];L^1_\mu)$, endowed with the norm $\sup_{t\in[0,T]}\norm{\cdot}_{L^1_\mu}$, is a metric space, it is enough to show that, for every sequence $(\varphi^n)_n$ with $\norm{\varphi^n}_\infty\leq1$, the sequence $(\mathcal V\varphi^n)$ admits a subsequence converging in $\rmC([0,T];L^1_\mu)$, that is, in $L^1_\mu$ uniformly in $t\in[0,T]$.

    \emph{Step 1: boundedness and equicontinuity.} As in Step~1 of the proof of Theorem~\ref{thm:fixed point}, for $\varphi\in L^\infty(0,T;\real^d)$ and $0\leq u\leq t\leq T$,
    $$
    \norm{(\mathcal V\varphi)_t}_{L^1_\mu}\leq\norm{\varphi}_\infty\norm{K}_{L^1(0,T)},\qquad \norm{(\mathcal V\varphi)_t-(\mathcal V\varphi)_u}_{L^1_\mu}\leq2\norm{\varphi}_\infty\int_0^{t-u}K(r)\de r .
    $$
    In particular, $\mathcal V$ is bounded, and all the functions $\mathcal V\varphi$ with $\norm{\varphi}_\infty\leq1$ share the modulus of continuity $h\mapsto2\int_0^hK(r)\de r$, which vanishes as $h\to0$ since $K\in L^1(0,T)$.

    \emph{Step 2: convergence for fixed $t$.} By the Banach-Alaoglu theorem \cite[Corollary~3.30]{Brezis2011}, $L^\infty(0,T)$ is the dual of the separable space $L^1(0,T)$, every bounded sequence in $L^\infty(0,T)$ has a weakly-$*$ convergent subsequence. Hence there exist $\varphi\in L^\infty(0,T;\real^d)$ with $\norm{\varphi}_\infty\leq1$ and a subsequence, still denoted $(\varphi^n)_n$, such that $\int_0^Tg(s)\varphi^n_s\de s\to\int_0^Tg(s)\varphi_s\de s$ for every $g\in L^1(0,T)$. Choosing $g(s)=e^{-\theta(t-s)}\mathbf 1_{[0,t]}(s)$, we get $(\mathcal V\varphi^n)_t(\theta)\to(\mathcal V\varphi)_t(\theta)$ for every $t\in[0,T]$ and $\theta\geq0$. To pass to the limit in $L^1_\mu$, we use dominated convergence. We have
    $$
    |(\mathcal V\varphi^n)_t(\theta)|\leq\int_0^te^{-\theta(t-s)}\de s\leq w(\theta):=T\land\theta^{-1},
    $$
    and $w$ belongs to $L^1_\mu$. Indeed, if $\theta T\leq1$ then $1-e^{-\theta T}\geq(1-e^{-1})\theta T$, while if $\theta T\geq1$ then $1-e^{-\theta T}\geq1-e^{-1}$; in both cases $(1-e^{-1})\,w(\theta)\leq(1-e^{-\theta T})/\theta$, with the convention that the right-hand side equals $T$ for $\theta=0$. Integrating against $\mu$, we obtain $\norm{w}_{L^1_\mu}\leq \frac{\norm{K}_{L^1(0,T)}}{1-e^{-1}}<\infty$.
    By dominated convergence, $(\mathcal V\varphi^n)_t\to(\mathcal V\varphi)_t$ in $L^1_\mu$ for every $t\in[0,T]$.

    \emph{Step 3: uniform convergence in $t$.} The convergence of Step~2 holds for each fixed $t$, while convergence in $\rmC([0,T];L^1_\mu)$ requires it uniformly in $t$. Since the functions $\mathcal V\varphi^n$ and $\mathcal V\varphi$ share the modulus of continuity of Step~1, pointwise convergence on the compact interval $[0,T]$ implies uniform convergence, by the usual $\eps/3$ argument. Hence $\mathcal V\varphi^n\to\mathcal V\varphi$ in $\rmC([0,T];L^1_\mu)$. Thus every sequence in the image of the unit ball of $B_b([0,T];\real^d)$ has a subsequence converging in $\rmC([0,T];L^1_\mu)$, so that this image has compact closure and $\mathcal V$ is compact.
\end{proof}

\section{Estimates for the Riccati equations}
\label{appendix:B}
This appendix contains the technical results underlying the fixed-point argument of Section~\ref{subsec:fixed point}; the bounds on $\hat{R}(\Gamma_t)^{-1}$ and $S(\Gamma_t)$ are also used in the mean field limit of Section~\ref{sec:propagation}. The analysis is based on the associated Riccati equations and focuses on the dependence of their solutions on the parameter $m$. In particular, we establish the Lipschitz continuity of the solution map $m \mapsto \Lambda^m$ from which the key properties of the mapping $\psi$ follow. We first record two uniform bounds on the operators entering the optimal feedback and the Riccati system.
\begin{lemma}
\label{lemma:hat_N_bdd}
    Suppose Assumptions~\ref{ass:kernel} and~\ref{ass:data} hold, and let $\Gamma$ be the solution of the first Riccati equation in \eqref{eq:Riccati}. Then
    \[
    c_R:=\sup_{t \in [0,T]}|\hat{R}(\Gamma_t)^{-1}| < \infty .
    \]
\end{lemma}
\begin{proof}
    By Definition~\ref{def:sol_riccati}, $\hat{R}(\Gamma_t)$ is invertible for every $t \in [0,T]$. Moreover, the map $[0,T] \ni t \mapsto \hat{R}(\Gamma_t) \in \real^{q \times q}$ is continuous, since for any $s, t \in [0,T]$
    \begin{align*}
        \left|\hat{R}(\Gamma_t) - \hat{R}(\Gamma_s)\right| \leq |F|^2 \|\Gamma_t - \Gamma_s\|_{L^1_{\mu \otimes \mu}},
    \end{align*}
    and the right-hand side tends to $0$ as $s \to t$, by continuity of $t \mapsto \Gamma_t$ in $L^1_{\mu \otimes \mu}$. Since matrix inversion is continuous on the set of invertible matrices, $t\mapsto\hat{R}(\Gamma_t)^{-1}$ is continuous on the compact interval $[0,T]$, hence bounded.
\end{proof}
\begin{lemma}\label{lemma:S_Gamma_bdd}
    Suppose Assumptions~\ref{ass:kernel} and~\ref{ass:data} hold. Then
    \[
    \sup_{t \in [0,T]}\norm{S(\Gamma_t)}_{L^\infty_\mu} \leq c_S:=|C|\,c_\Gamma+|F||D|\sup_{t\in[0,T]}\norm{\Gamma_t}_{L^1_{\mu\otimes\mu}},
    \]
    where $\Gamma_t$ is the solution of the first Riccati equation in \eqref{eq:Riccati} and $c_\Gamma$ is the constant of Proposition~\ref{prop:gamma}.
\end{lemma}
\begin{proof}
    By definition of $S$, for $\mu$-a.e.\ $\tau$,
    \begin{align*}
        |S(\Gamma_t)(\tau)| \leq |C| \norm{\Gamma_t(\cdot,\tau)}_{L^1_\mu} + |F||D|\norm{\Gamma_t}_{L^1_{\mu\otimes\mu}}.
    \end{align*}
    By the symmetry $\Gamma_t(\theta,\tau)=\Gamma_t(\tau,\theta)^\top$, we have $\norm{\Gamma_t(\cdot,\tau)}_{L^1_\mu}=\int_0^\infty|\Gamma_t(\tau,\theta)|\mu(\de\theta)\leq c_\Gamma$ for $\mu$-a.e.\ $\tau$, by Proposition~\ref{prop:gamma}. The supremum of $\norm{\Gamma_t}_{L^1_{\mu\otimes\mu}}$ over $[0,T]$ is finite by the continuity of $t\mapsto\Gamma_t$ in $L^1_{\mu\otimes\mu}$. The thesis follows by taking the essential supremum over $\tau \geq 0$.
\end{proof}

The next lemma is an elementary Lipschitz estimate for the affine map $h$ in its function arguments.
\begin{lemma}\label{lemma:h_lip}
Suppose Assumption~\ref{ass:kernel} holds. Let $m^1, m^2 \in \rmC([0,T];\mathbb{R}^d)$ and $\ell^1, \ell^2 \in L^1_\mu$. Then, for every $t \in [0,T]$ and every $G \in L^1_{\mu \otimes \mu}$, the following estimate holds
\begin{align}
    \big|h(t,G,\ell^1,m^1) - h(t,G,\ell^2,m^2)\big|
    \leq\Lip_{h,m}(G)\, \|m^1 - m^2\|_\infty +\Lip_{h,\ell}\, \|\ell^1 - \ell^2\|_{L^1_\mu}.
\end{align}
Here, $\mathrm{Lip}_{h,m}(G)$ is a Lipschitz constant depending on the matrices $D,F$ and on $G$ only through its $L^1_{\mu \otimes \mu}$-norm, while $\mathrm{Lip}_{h,\ell}$ depends only on the matrix $C$.
\end{lemma}
\begin{proof}
    The proof is a direct consequence of the definition of $h$ and triangular inequality. Indeed
    \[
    \big|h(t,G,\ell^1,m^1) - h(t,G,\ell^1,m^2)\big| = |F^\top \mu^{\otimes2}[G]D(m^1_t-m^2_t)| \leq |F||D|\norm{G}_{L^1_{\mu\otimes\mu}} \norm{m^1-m^2}_\infty,
    \]
    and
    \[
    \big|h(t,G,\ell^1,m^1) - h(t,G,\ell^2,m^1)\big| = |C^\top \mu[\ell^1-\ell^2]| \leq |C| \norm{\ell^1-\ell^2}_{L^1_\mu}.
    \]
\end{proof}
\begin{remark}
    Notice that the Lipschitz constant $\mathrm{Lip}_{h,m}(G)$, when evaluated at the solution $\Gamma_t$ of the first Riccati equation, is bounded, indeed
    \[
    \sup_{t \in [0,T]}\mathrm{Lip}_{h,m}(\Gamma_t) =|F||D| \sup_{t \in [0,T]}\norm{\Gamma_t}_{L^1_{\mu\otimes\mu}} =:\Lip_{h,m} < \infty
    \]
    by the continuity of $t\mapsto\Gamma_t$ in $L^1_{\mu\otimes\mu}$. We will use the notation $\mathrm{Lip}_{h,m}$ to denote this particular Lipschitz constant.
\end{remark}
Combining the previous bounds, we obtain the Lipschitz continuity of $\mathcal{R}_2$ evaluated along the solution $\Gamma$ of the first Riccati equation.
\begin{lemma}\label{lemma:R_2_lip}
    Suppose Assumption~\ref{ass:kernel} and Assumption~\ref{ass:data} hold. Let $m^1, m^2 \in \rmC([0,T];\mathbb{R}^d)$ and $\ell^1, \ell^2 \in L^1_\mu$. Then, for every $t \in [0,T]$, the following estimate holds
    \begin{equation}
        \norm{\mathcal{R}_2(t,\Gamma_t,\ell^1,m^1) - \mathcal{R}_2(t,\Gamma_t,\ell^2,m^2)}_{L^\infty_\mu}
    \leq\Lip_{\mathcal{R}_2}\{ \norm{m^1-m^2}_\infty + \norm{\ell^1-\ell^2}_{L^1_\mu} \},
    \end{equation}
    where $\Gamma_t$ is the solution of the first Riccati in \eqref{eq:Riccati}. The constant $\mathrm{Lip}_{\mathcal{R}_2}$ depends on the matrices of the model and on $\Gamma_t$ through its $\rmC([0,T];L^1_{\mu\otimes\mu})$-norm.
\end{lemma}
\begin{proof}
    By definition of $\mathcal{R}_2$, we get
    \begin{align*}
            &|\mathcal{R}_2(t,\Gamma_t,\ell^1,m^1)(\theta)-\mathcal{R}_2(t,\Gamma_t,\ell^1,m^2)(\theta)| \\[0.2cm]
        & \leq \Big(\left|\mu[\Gamma_t(\theta,\cdot)]\right| |B| + |D|^2\norm{\Gamma_t}_{L^1_{\mu\otimes\mu}}\Big)|m^1_t-m^2_t| + |S(\Gamma_t)(\theta)||\hat{R}(\Gamma_t)^{-1}||h(t,\Gamma_t,\ell^1,m^1)-h(t,\Gamma_t,\ell^1,m^2)| \\[0.2cm]
         & \leq \Big(c_\Gamma |B| + |D|^2\sup_{s\in[0,T]}\norm{\Gamma_s}_{L^1_{\mu\otimes\mu}} + c_Sc_R\Lip_{h,m}\Big) \norm{m^1-m^2}_\infty.
    \end{align*}
    Moreover,
    \begin{align*}
        &|\mathcal{R}_2(t,\Gamma_t,\ell^1,m^1)(\theta)-\mathcal{R}_2(t,\Gamma_t,\ell^2,m^1)(\theta)| \\[0.2cm]
        & \leq |B||\mu[\ell^1]-\mu[\ell^2]| + |S(\Gamma_t)(\theta)||\hat{R}(\Gamma_t)^{-1}||h(t,\Gamma_t,\ell^1,m^1)-h(t,\Gamma_t,\ell^2,m^1)| \\[0.2cm] 
        & \leq (|B| + c_Sc_R\mathrm{Lip}_{h,\ell})\norm{\ell^1-\ell^2}_{L^1_\mu}.
    \end{align*}
    
    Putting together the two computations and taking the essential supremum over $\theta$, we obtain the thesis with $\Lip_{\mathcal{R}_2}:=\max\big\{c_\Gamma |B| + |D|^2\sup_{s\in[0,T]}\norm{\Gamma_s}_{L^1_{\mu\otimes\mu}} + c_Sc_R\Lip_{h,m},\ |B| + c_Sc_R\Lip_{h,\ell}\big\}$.
    
\end{proof}
We can finally establish the main estimate of this appendix, the Lipschitz continuity of the solution map $m\mapsto\Lambda^m$, on which the Lipschitz property of $\psi$ (Lemma~\ref{lemma:psi_lip}) ultimately rests.
\begin{lemma}\label{lemma:Lambda_lip}
    Suppose Assumption~\ref{ass:kernel} and Assumption~\ref{ass:data} hold. Let $m^1, m^2 \in \rmC([0,T];\mathbb{R}^d)$ and let $\Lambda^1, \Lambda^2$ be the corresponding solutions to the second Riccati equation in \eqref{eq:Riccati} obtained with $m^1$ and $m^2$ respectively. Then, the following estimate holds
    \begin{equation}
        \sup_{t \in [0,T]} \norm{\Lambda^1_t-\Lambda^2_t}_{L^1_\mu} \leq c_K \norm{m^1-m^2}_\infty,
    \end{equation}
    where the constant $c_K$ depends on the matrices of the model, on $\Gamma_t$ through its $\rmC([0,T];L^1_{\mu\otimes\mu})$-norm, on the kernel $K$ through its $L^1(0,T)$-norm and on the $L^1(0,T)$-norm of its resolvent.
\end{lemma}
\begin{proof}
    By definition of solution to \eqref{eq:Riccati}, we get
    \begin{align*}
        \norm{\Lambda^1_t-\Lambda^2_t}_{L^1_\mu} & =  \int_0^\infty \left| \int_t^T e^{-\theta(s-t)}[\mathcal{R}_2(s,\Gamma_s,\Lambda^1_s,m^1)(\theta)-\mathcal{R}_2(s,\Gamma_s,\Lambda^2_s,m^2)(\theta)]\de s\right|\mu(\de\theta) \\[0.2cm]
        & \leq \int_0^\infty  \int_t^T e^{-\theta(s-t)}|\mathcal{R}_2(s,\Gamma_s,\Lambda^1_s,m^1)(\theta)-\mathcal{R}_2(s,\Gamma_s,\Lambda^1_s,m^2)(\theta)|ds\, \mu(\de\theta) \\
        & \qquad + \int_0^\infty  \int_t^T e^{-\theta(s-t)}|\mathcal{R}_2(s,\Gamma_s,\Lambda^1_s,m^2)(\theta)-\mathcal{R}_2(s,\Gamma_s,\Lambda^2_s,m^2)(\theta)|ds\, \mu(\de\theta) \\[0.2cm]
        & = \mathrm{\textbf{I}} + \mathrm{\textbf{II}}.
    \end{align*}
    By Fubini-Tonelli and Lemma \ref{lemma:R_2_lip},
    \begin{align*}
        \mathrm{\textbf{I}} & \leq \int_0^\infty  \int_t^T e^{-\theta(s-t)}\norm{\mathcal{R}_2(s,\Gamma_s,\Lambda^1_s,m^1)-\mathcal{R}_2(s,\Gamma_s,\Lambda^1_s,m^2)}_{L^\infty_\mu}ds\, \mu(\de\theta) \\[0.2cm]
        & \leq \int_0^\infty  \int_t^T e^{-\theta(s-t)}\Lip_{\mathcal{R}_2} \norm{m^1-m^2}_\infty ds\, \mu(\de\theta) =\Lip_{\mathcal{R}_2} \norm{m^1-m^2}_\infty\int_t^T K(s-t)\de s\\[0.2cm]
        & \leq\Lip_{\mathcal{R}_2}\norm{K}_{L^1(0,T)} \norm{m^1-m^2}_\infty.
    \end{align*}
    In an analogous way,
    \begin{align*}
        \mathrm{\textbf{II}} & \leq \int_0^\infty  \int_t^T e^{-\theta(s-t)}\norm{\mathcal{R}_2(s,\Gamma_s,\Lambda^1_s,m^2)-\mathcal{R}_2(s,\Gamma_s,\Lambda^2_s,m^2)}_{L^\infty_\mu}\de s\, \mu(\de\theta) \\[0.2cm]
        & \leq \int_0^\infty  \int_t^T e^{-\theta(s-t)}\Lip_{\mathcal{R}_2} \norm{\Lambda^1_s-\Lambda^2_s}_{L^1_\mu} \de s\, \mu(\de\theta) =\Lip_{\mathcal{R}_2}\int_t^T K(s-t)\norm{\Lambda^1_s-\Lambda^2_s}_{L^1_\mu}\de s.
    \end{align*}
    Putting all together, we get
    \begin{align*}
        \norm{\Lambda^1_t-\Lambda^2_t}_{L^1_\mu} \leq\Lip_{\mathcal{R}_2}\left\{ \norm{K}_{L^1(0,T)} \norm{m^1-m^2}_\infty + \int_t^T K(s-t)\norm{\Lambda^1_s-\Lambda^2_s}_{L^1_\mu}\de s\right\}.
    \end{align*}
    The inequality above is of backward type. Setting $x(u) := \norm{\Lambda^1_{T-u}-\Lambda^2_{T-u}}_{L^1_\mu}$ for $u\in[0,T]$ and changing variables $s = T-v$, it reads
    \[
    x(u) \leq \Lip_{\mathcal{R}_2}\norm{K}_{L^1(0,T)} \norm{m^1-m^2}_\infty + \Lip_{\mathcal{R}_2}\int_0^u K(u-v)\,x(v)\de v, \qquad u \in [0,T].
    \]
    By Assumption \eqref{ass:kernel} we see that $K\geq 0$ and $K \in L^1(0,T)$, therefore the resolvent $r$ of $-\Lip_{\mathcal{R}_2}K$ is non-positive and belongs to $L^1(0,T)$. The convolutional Gronwall inequality (Theorem~\ref{thm:gronwall}) then gives
    \[
     \norm{\Lambda^1_t-\Lambda^2_t}_{L^1_\mu} \leq c_K \norm{m^1-m^2}_\infty, \qquad c_K := \Lip_{\mathcal{R}_2}\norm{K}_{L^1(0,T)}\big(1+\norm{r}_{L^1(0,T)}\big),
    \]
    and we conclude by taking the supremum over $t \in [0,T]$.
\end{proof}

\begin{remark}
\label{rmk:horizon dependence}
    The constants of this appendix are computed on a fixed horizon $T$. They depend only on $\mu$, $T$ and the matrices $B,C,D,F,Q,R$, and they are non-decreasing in $T$. Indeed, write $\Gamma^{(T)}$ for the solution of the first Riccati equation on $[0,T]$. Its coefficients depend neither on time nor on $g_0,\beta,\gamma,L,m$, so that, if $T\leq T'$, the function $t\mapsto\Gamma^{(T')}_{t+T'-T}$ solves the same equation on $[0,T]$; by uniqueness, $\Gamma^{(T)}_t=\Gamma^{(T')}_{t+T'-T}$ for $t\in[0,T]$. Hence, taking $c_\Gamma$ as the smallest constant in Proposition~\ref{prop:gamma}, the quantities $c_\Gamma$, $\sup_{t\in[0,T]}\norm{\Gamma_t}_{L^1_{\mu\otimes\mu}}$, $c_R$, $c_S$, $\Lip_{h,m}$ and $\Lip_{\mathcal R_2}$ are non-decreasing in $T$, while $\Lip_{h,\ell}=|C|$ does not depend on $T$. The same holds for $c_K$, since $|r|=\sum_{n\geq1}(\Lip_{\mathcal R_2}K)^{*n}$ has an $L^1(0,T)$-norm which is non-decreasing both in $\Lip_{\mathcal R_2}$ and in $T$. Consequently $\Lip_\psi$ (Lemma~\ref{lemma:psi_lip}) and $\rho_T=\Lip_\psi\norm{K}_{L^1(0,T)}$ are non-decreasing in $T$. Moreover $\rho_T\to0$ as $T\downarrow0$, because $\Lip_\psi$ stays bounded for $T\leq1$ and $\norm{K}_{L^1(0,T)}\to0$.
\end{remark}

\section{Auxiliary estimates for the mean field limit}
\label{appendix:convergence tools}
This appendix collects two estimates used in the convergence analysis of Section~\ref{sec:propagation}: a law of large numbers quantifying the fluctuations of empirical means, and a moment estimate for lifted stochastic Volterra dynamics driven by general coefficients.

\begin{lemma}
\label{lemma:CLT}
    Let $(U^i)_{i \geq 1}$ be a sequence of $\real^d$-valued i.i.d.\ stochastic processes such that $\sup_{t \in [0,T]} \E[|U^i_t|^r] \leq c_U$ for some $r\geq 2$ and some positive constant $c_U$. Then, for every $N \geq 1$ and $1 \leq p \leq r$ we have
    \begin{equation*}
        \sup_{t \in [0,T]}\E\left[ \left| \frac{1}{N}\sum_{j=1}^N U^j_t - \E[U^1_t] \right|^p \right] \leq \frac{c}{N^{p/2}},
    \end{equation*}
    where $c$ is a constant depending only on $p$, $r$ and $c_U$.
\end{lemma}
\begin{proof}
    Fix $t \in [0,T]$. An application of the Marcinkiewicz-Zygmund inequality (see e.g.~\cite[Section~10.3, Theorem~2]{Chow_Teicher_prob_book}) gives
    \begin{align*}
        \E\left[ \left| \frac{1}{N}\sum_{j=1}^N U^j_t - \E[U^1_t] \right|^p \right] & \leq  \frac{c_p}{N^{p/2}}\, \E\left[ \left| \frac{1}{N}\sum_{j=1}^N |U^j_t - \E[U^1_t]|^2 \right|^{p/2} \right]  \\[0.2cm]
        & \leq \frac{c_p}{N^{p/2}}\, \E\left[  \frac{1}{N}\sum_{j=1}^N |U^j_t - \E[U^1_t]|^p  \right] = \frac{c_p}{N^{p/2}}\,\E\left[ |U^1_t-\E[U^1_t]|^p \right],
    \end{align*}
        where in the second line we have used Jensen's inequality, which requires $p\geq2$, and the fact that the $U^j$ are identically distributed. Since $\E[ |U^1_t-\E[U^1_t]|^p ] \leq \tilde{c}_p\, \E[ |U^1_t|^p ] \leq \tilde{c}_p c_U$, this proves the claim for $2\leq p\leq r$. For $1\leq p<2$ it suffices to apply the case $p=2$ and the inequality $\E[|Z|^p]\leq\E[|Z|^2]^{p/2}$.
\end{proof}

\begin{proposition}
\label{prop:usual machinery}
    Suppose Assumption~\ref{ass:kernel} holds. Let $p \geq 2$ and let $\Delta^1, \Delta^2$ be two progressively measurable processes such that $\sup_{t\in [0,T]} \E\left[ \left| \Delta^1_t \right|^p + \left| \Delta^2_t \right|^p \right] < \infty$, and define
    \[
    Z_t(\theta) := \int_0^t e^{-\theta(t-s)}\Delta^1_s\,\de s + \int_0^t e^{-\theta(t-s)}\Delta^2_s\,\de W^i_s, \quad t \in [0,T], \mu\text{-a.e.}    
    \]
    Then there exists a kernel $\widehat{K} \in L^1(0,T)$ such that $-\widehat{K}$ has a non-positive resolvent $\widehat{r}$ for which it holds
    \begin{equation}
    \label{eq:usual machinery}
        \mathbb{E}\left[ \norm{Z_t}_{L^1_\mu}^p \right] \leq c\int_0^t \widehat{K}(t-s)\,\E\left[|\Delta^1_s|^p+|\Delta^2_s|^p\right]\de s, \quad t \in [0,T],
    \end{equation}
    for some positive constant $c$ depending only on $\mu$, $p$ and $T$.
\end{proposition}
\begin{proof}
        Following the proof of \cite[Theorem 4.1]{LQ_Volterra_AJMP}, we apply Jensen's inequality on the probability measure $\frac{1 \land \theta^{-1/2}}{\int_0^\infty 1 \land \tau^{-1/2}\mu(d\tau) }\mu(\de\theta)$ (the constant $c$ may change from line to line)
    \begin{align*}
       \norm{Z_t}^p_{L^1_\mu} & = \left(\int_0^\infty 1 \land \tau^{-1/2}\mu(\de\tau)\right)^p\left( \int_0^\infty \frac{1}{1 \land \theta^{-1/2}}|Z_t(\theta)| \frac{1 \land \theta^{-1/2}\mu(\de \theta)}{\int_0^\infty 1 \land \tau^{-1/2}\mu(\de\tau)}  \right)^p \\[0.2cm]
       & \leq c \int_0^\infty \left(1 \vee \theta^{\frac{p-1}{2}}\right) \left| \int_0^t e^{-\theta(t-s)}(\Delta^1_s\de s+ \Delta^2_s\de W_s^i) \right|^p \mu(\de\theta) \\[0.2cm]
       & \leq c \int_0^\infty \left(1 \vee \theta^{\frac{p-1}{2}}\right)\left\{\left| \int_0^t e^{-\theta(t-s)}\Delta^1_s\de s\right|^p +\left| \int_0^t e^{-\theta(t-s)}\Delta^2_s\de W^i_s \right|^p\right\}\mu(\de\theta).
    \end{align*}
    Taking the expectation and using the BDG inequality, we get
    \begin{align*}
    \mathbb{E}\!\left[ \|Z_t\|_{L^1_\mu}^p \right]
    &\leq c \int_0^\infty \!\left(1 \vee \theta^{\frac{p-1}{2}}\right)
    \Bigg\{
        \mathbb{E}\!\left[
            \left| \int_0^t e^{-\theta(t-s)} \Delta^1_s \,\de s\right|^p\right]  %
        + \mathbb{E}\!\left[
            \left| \int_0^t e^{-2\theta(t-s)} |\Delta^2_s|^2 \, \de s\right|^{p/2}\right]\Bigg\}\, \mu(\de\theta) \\[0.2cm]
            & \leq c \int_0^\infty \!\left(1 \vee \theta^{\frac{p-1}{2}}\right)
    \Bigg\{
        \mathbb{E}\!\left[
            \left| \int_0^t e^{-2\theta(t-s)} |\Delta^1_s|^2\,\de s\right|^{p/2}\right]  %
        + \mathbb{E}\!\left[
            \left| \int_0^t e^{-2\theta(t-s)} |\Delta^2_s|^2 \, \de s\right|^{p/2}\right]\Bigg\}\, \mu(\de\theta), 
    \end{align*}
    where in the last inequality we have used Jensen's inequality on the probability measure $\frac{\de s}{t}$ on $[0,t]$. Now, we use Jensen again on the probability measure $\frac{e^{-2\theta(t-s)}}{\int_0^t e^{-2\theta(t-u)}\de u}\de s$ to get
    \begin{align*}
        \mathbb{E}\!\left[ \|Z_t\|_{L^1_\mu}^p \right] & \leq c\int_0^\infty \!\left(1 \vee \theta^{\frac{p-1}{2}}\right)\left( \int_0^t e^{-2\theta(t-u)}\de u \right)^{\frac{p}{2}-1}  \int_0^t e^{-2\theta(t-s)} \E\left[|\Delta^1_s|^p+|\Delta^2_s|^p\right]\de s\,  \mu(\de\theta) \\[0.2cm]
        & = c \int_0^t \E\left[|\Delta^1_s|^p+|\Delta^2_s|^p\right] f(t-s)\,\de s,
    \end{align*}
    where
    \begin{align*}
        f(t-s) & = \int_0^\infty \left(1 \vee \theta^{\frac{p-1}{2}}\right)\left( \int_0^t e^{-2\theta(t-u)}\de u \right)^{\frac{p}{2}-1}e^{-2\theta(t-s)}  \mu(\de\theta) \\[0.2cm]
        & = \int_0^\infty  \left(1 \vee \theta^{\frac{p-1}{2}}\right) \left( \frac{1-e^{-2\theta t}}{2\theta} \right)^{\frac{p}{2}-1} e^{-2\theta(t-s)}\mu(\de\theta). 
    \end{align*}
    Using the inequality $
        \frac{1-e^{-2\theta t}}{2\theta} \leq \frac{1}{2}(1\vee 2t)(1 \land \theta^{-1})$,
    we get
    \begin{align*}
        f(t-s) & \leq c  \int_0^\infty \left(1 \vee \theta^{\frac{p-1}{2}}\right)  \left( 1 \land \theta^{1-\frac{p}{2}} \right) e^{-2\theta(t-s)}\mu(\de\theta) \\[0.2cm]
        & = c \int_0^\infty \left(1 \vee \theta^{\frac{1}{2}}\right)e^{-2\theta(t-s)}\mu(\de\theta) \leq c\left( 1 \vee \frac{1}{\sqrt{2e(t-s)}} \right) \int_0^\infty e^{-\theta(t-s)}\mu(\de\theta),
    \end{align*}
    since $\mu$-a.e. $e^{-\theta(t-s)}(1 \vee \theta^{\frac{1}{2}}) \leq  1 \vee \frac{1}{\sqrt{2e(t-s)}} $. Thus, defining the kernel
    \[
    \widehat{K}(t) := \left( 1 \vee \frac{1}{\sqrt{2et}} \right) K(t),
    \]
    we have that $\widehat{K} \in L^1(0,T)$ by the integrability conditions on $\mu$ (Proposition~\ref{prop:mu_properties}, applied with $\alpha=1/2$) and that $-\widehat{K}$ has a non-positive resolvent $\widehat{r}$. The conclusion \eqref{eq:usual machinery} follows.
\end{proof}

\bibliography{bibliography}
\bibliographystyle{plain}
\end{document}